\documentclass[11pt, a4paper]{article}
\usepackage{amsmath,amsthm,amssymb,amsfonts}
\usepackage{a4wide}

\newtheorem{theorem}{Theorem}[section]
\newtheorem{lemma}[theorem]{Lemma}
\newtheorem{proposition}[theorem]{Proposition}

\newtheorem{remark}[theorem]{Remark}

\def\Xint#1{\mathchoice
{\XXint\displaystyle\textstyle{#1}}%
{\XXint\textstyle\scriptstyle{#1}}%
{\XXint\scriptstyle\scriptscriptstyle{#1}}%
{\XXint\scriptscriptstyle\scriptscriptstyle{#1}}%
\!\int}
\def\XXint#1#2#3{{\setbox0=\hbox{$#1{#2#3}{\int}$ }
\vcenter{\hbox{$#2#3$ }}\kern-.6\wd0}}

\def\dashint{\Xint-}

\usepackage{graphicx}
\usepackage[usenames,dvipsnames]{color}
\usepackage{hyperref}
\hypersetup{colorlinks=true, pdfstartview=FitV, linkcolor=black, citecolor=blue, urlcolor=blue}

\makeatletter

\@addtoreset{equation}{section}
\makeatother

\newcommand{\R}{{\mathbb R}}
\newcommand{\mO}{\mathcal{O}}

\newcommand{\mS}{\mathcal{S}}

\newcommand{\mJ}{\mathcal{J}}

\newcommand{\mR}{\mathcal{R}}
\newcommand{\M}{\mathcal{M}}
\newcommand{\rp}{\mathrm{p}}
\renewcommand{\d}{\mathrm{d}}

\newcommand{\na}{{\nabla}}

\newcommand{\eps}{{\varepsilon}}

\def\div{ \mathrm{div} \,}

\def\Xint#1{\mathchoice
{\XXint\displaystyle\textstyle{#1}}%
{\XXint\textstyle\scriptstyle{#1}}%
{\XXint\scriptstyle\scriptscriptstyle{#1}}%
{\XXint\scriptscriptstyle\scriptscriptstyle{#1}}%
\!\int}
\def\XXint#1#2#3{{\setbox0=\hbox{$#1{#2#3}{\int}$ }
\vcenter{\hbox{$#2#3$ }}\kern-.6\wd0}}

\def\dashint{\Xint-}

\usepackage{mathtools}
\usepackage{pst-all}
\usepackage{etex}
\usepackage{float} 
\usepackage{appendix}

\usepackage{tikz,pgfplots}
\usetikzlibrary{patterns,matrix,backgrounds, calc, fit,decorations.pathreplacing}

\title{Derivation of an effective rheology for dilute suspensions of micro-swimmers.} 

\author{Alexandre Girodroux-Lavigne}

\begin{document}
\maketitle
\begin{abstract}
We provide a rigorous derivation of an effective rheology for dilute viscous suspensions of self-propelled particles. We derive at the limit an effective Stokes system with two additional terms. The first term corresponds to an effective viscosity that matches the famous Einstein's formula $\mu_{eff}=(1+ \frac{5}{2} \lambda)\mu$. The second additional term derived is an effective \textit{active stress} $\sigma_1$ that appears as a forcing term $\lambda \, \div \sigma_1$ in the fluid equation. We then investigate the effects of this term depending on the micro-swimmers nature and their distribution in the fluid.
\end{abstract}
\medskip

\noindent
\begin{MSC}
35Q35, 76T20, 76D07.
\end{MSC}
\section{Introduction}

Active suspensions analysis is a subject of growing interest. Broadly speaking, an active suspension is a large collection of swimming particles immersed into a fluid. These particles, sometimes referred to as \textit{micro-swimmers}, convert chemical energy into mechanical energy through a swimming mechanism. Classic examples and applications arise in biological fluids (cytoplasm, spermatozoon), micro-algae or bacteria suspensions \cite{sokolov2009enhanced,kim2007use, yasa2018microalga} and artificial micro-machines \cite{paxton2004catalytic}. Moreover, active fluids display unique experimental behaviors like pattern formations, chemotaxis, unsteady whirls and jets, exotic rheology and others that still lack theoretical understanding. \medskip

 Specifically, the \textit{rheology of active fluids} has been intensively studied in the last years. Depending  the nature of the particles, the suspension viscosity can either \textit{increase} or \textit{decrease} due to the particles activity. In \cite{hatwalne2004rheology}, a coarse-grained simple model is proposed to predict this fluid viscosity modification. A few years later, Aranson and Sokolov measured for the first time the shear viscosity of a bacterial suspension (\textit{Bacillus Subtilis}) and witnessed a drastic viscosity reduction, which pointed out a clear contrast when compared with a solution of immobile particles, see \cite{sokolov2009reduction}. Other experiments in that sense were achieved in \cite{rafai2010effective, gachelin2013non,lopez2015turning,chui2021rheology}. Following Aranson and Sokolov results, many theoretical studies proposed models to explain these viscosity changes. In all of these studies, some \textit{anisotropy on the particles orientation distribution} is needed to observe this deviated rheology. The origin of this anisotropy depends on the model studied, and is sometimes directly assumed. A broad panorama of the literature is given below.

\begin{itemize}
    \item A simplified two-dimensional model studied by Haines and al in \cite{haines2008effective} represents particles self-propulsion effect on the fluid by a Dirac source term located near each particle's boundary. On the other side, the balance of forces on each particle is modified through this  pushing (or pulling) force. Assuming a background shear-flow and the alignment of particles in the flow, they pointed out a decrease of the effective shear viscosity. This description of a micro-swimmer will be our starting point and we refer to Subsection \ref{subsection::our_model} for more details.
    
s    \item The previous model was adapted to the three-dimensional case in \cite{haines2009three, haines2011effective} to study ellipsoidal bacteria. There,  \textit{tumbling} and \textit{rotational diffusion} of the bacteria are taken into account, which allows to compute a solution of the steady Fokker-Planck equation solved by the probability density of the particles. This solution leads to some anisotropy in the particles orientation distribution, which enables to compute the effective viscosity changes in both pure shear and pure straining flows.
    
    \item In \cite{gyrya2011effective}, the active particle membrane is separated into two parts. The fluid sticks to the first part, and a tangential traction modelling the activity is prescribed on the second part. In the \textit{dilute regime}, this model leads to a uniform orientation distribution of the particles and it is shown that the obtained effective viscosity is not affected by self-propulsion. In the \textit{semi-dilute regime}, swimmer-swimmer interactions are taken into account, disrupting the rotational behavior of the particles.  Numerical simulations then show a global decrease of the effective shear viscosity.
    
    \item A new way to describe the particle flagellum is proposed in \cite{decoene2011microscopic}. The self-propulsion exerts a force on a very elongated ellipse outside the particle that represents its flagellum. This modeling was then used in \cite{decoene2019modelisation, chibbaro2021irreversibility} to numerically compute the modification of the effective shear viscosity in the semi-dilute regime.

     \item In the work of Pedley \cite{ pedley2016spherical}, a new model of active particles is introduced: the \textit{steady squirmer}. This particle moves through mucus absorption and secretion which results in a prescribed relative velocity on the particle surface. The rheology of a \textit{squirmer} suspension is then analysed in \cite{ishikawa2007rheology, ishikawa2021rheology}. Using two numerical methods, \textit{Stokesian dynamics} and \textit{lubrication theory}, the effective shear viscosity of the suspensions is computed in the semi-dilute regime. In this study, the anisotropy is a consequence of two effects: particles interaction and bottom heaviness (along with gravity) of the particles.
    \item It is well-known that the disturbance flow induced by a single active particle in a simple Stokes flow behaves as a force dipole in the far field\footnote{Note that this approximation was experimentally justified in reference \cite{drescher2011fluid}.}.  Based on this dipole representation, Saintillan introduced a kinetic description to understand active suspensions in \cite{saintillan2008instabilities}. Assuming some rotational diffusion, the probability density solution of the Fokker-Planck equation is numerically computed in \cite{saintillan2010dilute} and used to show shear thinning and shear thickening depending on the particles nature. Further studies on this model were achieved in \cite{saintillan2010extensional,chen2013global,saintillan2015theory,alonso2016microfluidic} and we refer to \cite{saintillan2018rheology} for a complete review.
\end{itemize}
 A few other articles deal with hydrodynamic-interactions and their impact on rheology in semi-dilute regime, see \cite{ryan2011viscosity,potomkin2016effective,zhang2021effective}.  More complicated settings were also recently studied, like micro-swimmers with bending flagella \cite{potomkin2017flagella} or micro-swimmers in \textit{nematic liquid crystals \cite{chi2020surface}, \textit{viscoelastic fluids} \cite{li2021microswimming} or in \textit{emulsions} \cite{favuzzi2021rheology}.} \black

\subsection{Model and assumptions}
\label{subsection::our_model}
We consider $N$ spherical particles $B_i:=\overline{B(x_i,a)}$ of small radius $a$ and centered at $x_{i}$ for $ 1 \leq i \leq N$, all included in a fixed compact of $\R^3$ independant of $N$. Note that the balls centers $x_i$ and radius $a$ depend on $N$ but we omit it in the notation. These particles are immersed in a Stokes flow of constant viscosity $\mu$. Each particle $B_i$ is being pushed (or pulled) by a swimming mechanism along the direction $\rp_i$ in the three-dimensional sphere  $\mS^2$, using for example a flagellum. To take this phenomenon into account, we represent this self-propulsion using a force monopole as in the model presented in \cite{haines2009three,haines2011effective}. Denoting 
\begin{align*}
    \Omega_N := \R^3 \setminus \cup_{i=1}^N B_i, \quad \text{the fluid domain,}
\end{align*}
the equations for the fluid velocity $u_N$ and its associated pressure $p_N$ are:
\begin{equation}
    \left\{
      \begin{aligned}
         - \mu \Delta u_N + \na p_N & = g_N - k_f \sum_{i=1}^N \delta(x - x_{f,i})  \, \rp_i   &&\text{ in } \Omega_N,\\
         \div  u_N &  = 0  &&\text{ in } \Omega_N,\\
         u_N(x) & = U_i +V_i \times (x-x_i) &&\text{ in } B_i,\\
         \underset{|x| \rightarrow + \infty}{ \lim} u_N(x) &  = 0,
      \end{aligned}
    \right.
    \label{system::main}
\end{equation}
where $g_N$ represents some forcing that converges toward some $g$ as $N$ grows. The action of particle $B_i$ on the fluid is a point force with intensity $k_f \in \R$ concentrated at point $x_{f,i}:= x_i + a \beta \rp_i$ with $\beta>1$, corresponding to a point close to the particle. The no-slip boundary condition on each ball $B_i$ involves the translation and rotation velocities $(U_i,V_i)$ on $B_i$. These constant vector fields are unknown but are determined through the balance of forces and torques for each particle $B_i$:
\begin{align*}
    & \int_{\partial B_i}\sigma(u_N,p_N){} n + k_f \rp_i  + \int_{B_i} g_N = 0\\
    &  \int_{ \partial B_i} \sigma(u_N,p_N) n \times (x-x_i) + \int_{B_i} g_N \times (x-x_i)=0,
\end{align*}
where $\sigma(u,p):= 2 \mu D(u) - p Id$ is the standard Cauchy stress tensor and $n$ is the normal vector pointing outward $B_i$. \medskip

\noindent \textbf{Assumptions.} The intensity $k_f$ of the pushing/pulling is chosen to scale like the particle effective surface, e.g:
\begin{equation}
    \label{assumption::force}
    k_f:= \alpha \pi \mu a^2
\end{equation}
where the sign of $\alpha \in \R$ determines whether it is a pulling or pushing swimmer. Indeed, if it is a pusher particle, which propels itself by exerting an outward force near its tail, we have $\alpha<0$ as the fluid is being pushed away from the particle. For a puller particle that swims using his arms to pull itself forward, one has $\alpha>0$. In the figures below, the small arrows indicate the direction of the pushing/pulling force $- k_f \delta_{x_{f,i}} \rp_i$ and the bigger ones marks the self-propulsion $k_f \rp_i$ on the particle.
\definecolor{gray1}{gray}{0.85}
\definecolor{gray2}{gray}{0.67}
\definecolor{gray3}{gray}{0.55}

\begin{minipage}{0.45\textwidth}
    \begin{figure}[H]
    \centering
    \includegraphics[scale=0.4]{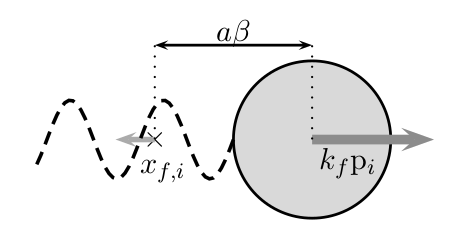}
    \caption{Pusher type particle, $k_f <0$.}
    \label{Geometry 1}
\end{figure}
\end{minipage}
\hspace{4ex} 
\begin{minipage}{0.45\textwidth}
\begin{figure}[H]
    \centering
   \includegraphics[scale=0.4]{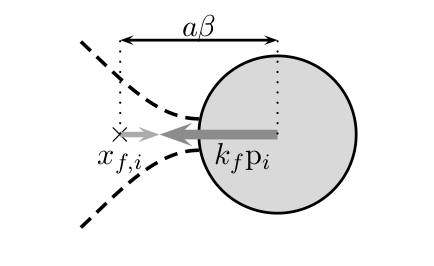}
    \caption{Puller type particle, $k_f>0$.}
    \label{Geometry 2}
\end{figure}
\end{minipage}\\ \medskip

\noindent We assume the following \textit{empirical measure convergence}:
\begin{equation}
    \tag{H1}
    f_N :=\frac{1}{N} \sum_{i= 1}^N \delta_{(x_i, \rp_i)} \underset{N \rightarrow + \infty}{\rightharpoonup} f(x,\rp) \, \d x \d \rp \quad \text{in the sense of measures,}
    \label{assumption::empirical_measure}
\end{equation}
where $f \in L^\infty(\R^3 \times \mS^2, \R^+)$ is a bounded density. Furthermore, we assume there exists a smooth bounded open set $\mO$ of $\R^3$ such that, for any $\rp \in \mS^2$, function $f(\cdot,\rp)$ has a support in $\overline{\mO}$. For simplicity, we suppose that $\mO$ is of volume one. Denoting 
\begin{align*}
    \rho(x):= \int_{\mS^2} f(x,\rp) \, \d \rp, \quad \text{the particles spatial density, }
\end{align*}
we easily show that $\eqref{assumption::empirical_measure}$ implies
\begin{align}
    \tag{H1'}
    & \delta_N:=\frac{1}{N} \sum_{i= 1}^N \delta_{x_i} \underset{N \rightarrow + \infty}{\rightharpoonup} \rho(x) \, \d x \quad \text{in the sense of measures,}
    \label{assumption::empirical_measure_spatial}\\
    & \mathrm{support}(\rho) \subset \overline{\mO}, \quad \rho \in L^\infty(\R^3).  \nonumber
\end{align}
We  investigate the \textit{dilute suspensions regime,} which means the volume fraction  $\lambda:= \frac{4}{3} \pi a^3 N$ is small, but independent of $N$. Hence, the particles radius $a$ scales like $N^{- \frac{1}{3}}$.\medskip

\noindent We also make the following \textit{separation assumption}: there exists $c>0$ such that for any $N>0$, we have
 \begin{equation}
    \tag{H2}
     \underset{i \neq j}{\min} \, |x_i - x_j| \geq c N^{-1/3}.
    \label{assumption::separation}
\end{equation}
Our analysis in Section \ref{section::Relaxed Assumption} demonstrates how this assumption may be relaxed, using arguments from $\cite{GVRH}$. Note that to ensure that all the forces application points $x_{f,i}$ are indeed in the fluid domain $\Omega_N$, we need to verify $ \underset{i \neq j}{\min} \, |x_i - x_j| \geq 2 \beta a$. In the light of assumption \eqref{assumption::separation}, it is enough to satisfy the following inequality:
\begin{equation}
    1<\beta< \frac{c}{2} \Big(\frac{4 \pi}{3} \Big)^{1/3} \lambda^{-1/3},
    \label{assumption::separation_beta}
\end{equation}
which is automatically satisfied for small enough $\lambda$.\medskip

\noindent \textbf{Aim of the study.} As $N \rightarrow + \infty$, the hope is to replace \eqref{system::main} by a Stokes system without particles, with a new viscosity $\mu'$ in the domain $\mO$ and a new source term $f'$, and find an asymptotic development when $\lambda$ is small:
\begin{equation}
    \mu' = \mu Id +  \lambda  \mu_1 I_d + \lambda^2 \mu_2+ \dots \quad  \text{and} \quad f' = g+  \lambda f_1 + \lambda^2 f_2 + \dots
    \label{intro::dev effectif}
\end{equation}
More precisely, we want to show that, for $N$ very large and at small volume fraction $\lambda$, the solution $u_N$ of \eqref{system::main} is $o(\lambda)$-close to the solution $\bar{u}$ of the following effective fluid model
\begin{equation}
    \left\{
      \begin{aligned}
         - 2 \mu \, \div ([ 1+\frac{5}{2} \rho \lambda ] D(\bar{u}))+ \na \bar{p} & = g + \lambda \, \div{} \sigma_1 && \ \text{in} \ \R^3,\\
         \div \bar{u} & = 0 && \text{ in } \R^3,\\
         \underset{|x| \rightarrow + \infty}{ \lim} \bar{u}(x)  & = 0
      \end{aligned}
    \right.
    \label{system::intro_effective_model}
\end{equation}
where $\sigma_1$ is a matrix-valued function defined on the whole space by
\begin{equation}
    \sigma_1(x)= \alpha \mJ \int_{\mS^2} (\rp \otimes \rp-\frac{1}{3} Id) f(x,\rp) \d \rp, \quad \text{for some} \quad \mJ >0 .
    \label{first def:: sigma_1}
\end{equation}
In the above system, the constant fluid viscosity $\mu$ has been replaced by $(1+\frac{5}{2} \rho \lambda) \mu$. This new viscosity corresponds to the famous Einstein first-order expansion of the effective viscosity induced by a rigid spheres suspension. This effective viscosity was introduced for the first time in \cite{Einstein}, and was since then the inspiration for many studies, see \cite{Haines&Mazzucato,GVH,HillairetWu,DueGloria, GVRH} among many. In our model, this effective viscosity encodes the particles influence as if they were passive rigid bodies. \medskip

\noindent The term $\sigma_1$ can be understood as an additional stress due to the particles self-propulsion. A similar term called \textit{active stress tensor} arised in many publications such as \cite{saintillan2008instabilities,saintillan2010dilute, saintillan2018rheology, rafai2010effective, potomkin2016effective, alonso2016microfluidic,degond2019coupled} and is reminiscent of Batchelor's primal work \cite{batchelor1971stress}. In most of them, this active stress tensor is due to some averaging of the force dipoles produced by each micro-swimmer. However, no rigorous derivation or convergence toward effective models were proposed. Indeed, the history of the mathematical study of active suspensions is scarcer than for passive suspensions. \medskip

\noindent Our objective is two-fold:
\begin{itemize}
    \item to justify the materialization of particles self-propulsion by this active stress tensor $\sigma_1$ and to understand its effects on the system
    \item rigorous homogenization theory for active suspensions has yet to be performed. Proving that $u_N$ is close as $N \rightarrow + \infty$ to the effective fluid model satisfied by $\bar{u}$, up to an error of order $\lambda$, is a first step in this direction. 
\end{itemize} \medskip

\noindent \textbf{Main results.} One of the main difficulties in studying the limit of system \eqref{system::main} is the presence of the irregular Dirac source terms. The first issue is to bound the solution in a suitable space. This is the purpose of the following Theorem

\begin{theorem}
\label{thm::u_N_bornée}
Let $\lambda>0$.  Assume that $g_N \rightarrow g$  in $L^{\frac{6}{5}}(\R^3) \cap L^{3+\eps}(\R^3)$, for some $\eps >0$, and assume that \eqref{assumption::empirical_measure} and \eqref{assumption::separation} hold. Then for all compact $K$, there exists a constant $C>0$ independent of $N$ such that 
\begin{align*}
   ||u_N||_{L^p(K)} \leq C \quad \text{for any} \quad 1<p< \frac{3}{2}.
\end{align*}
\end{theorem}
\noindent Note that this Theorem allows us to define an accumulation point of the sequence $(u_N)_{N}$ in $L^p_{loc}(\R^3)$ for the same range of $p$.

\begin{theorem}
Let $\lambda>0$. Assume that $g_N \rightarrow g$  in $L^{\frac{6}{5}}(\R^3) \cap L^{3+\eps}(\R^3)$, for some $\eps >0$, and assume that \eqref{assumption::empirical_measure} and \eqref{assumption::separation} hold. Let $u_\lambda$ be an accumulation point of the sequence $(u_N)_{N}$ in  $L^p_{loc}(\R^3)$, then the field $u_\lambda$ solves in the sense of distributions
\begin{equation}
    \left\{
      \begin{aligned}
         - 2 \mu \, \div ([1+\frac{5}{2}\rho \lambda] D(u_\lambda)) + \na p_\lambda & = g +\lambda \, \div{} \sigma_1 +  \mR_\lambda & \ \text{in} \ \R^3,\\
         \div  u_\lambda &  = 0 & \ \text{in} \ \R^3,\\
         \underset{|x| \rightarrow + \infty}{ \lim} u_\lambda(x)  & = 0,
      \end{aligned}
    \right.
    \label{system::effective_main_theorem} 
\end{equation}
where $\sigma_1$ is the trace-free matrix defined in \eqref{first def:: sigma_1} with $\mJ:=\frac{3 \alpha \mu}{4}\Big(\beta- \frac{5}{2} \beta^{-2} + \frac{3}{2} \beta^{-4} \Big)$ and $\mR_\lambda$ satisfies for every $q \geq 3$ 
\begin{align*}
    & |\langle  \mR_\lambda | \phi \rangle| \leq C \lambda^{\frac{5}{3}} ||D(\phi)||_{L^q(\R^3)}, \quad \forall \phi \in \dot{H}^1(\R^3) \cap \dot{W}^{1,q}(\R^3) \text{ divergence-free},
\end{align*}
and where we denote:
\begin{align*}
    \dot{H}^1(\R^3) & := \{h \in W^{1,2}_{loc}(\R^3), \ \na h \in L^2(\R^3)  \}\\
    \dot{W}^{1,q}(\R^3)& := \{h \in W^{1,q}_{loc}(\R^3), \ \na h \in L^q(\R^3) \}.
\end{align*}
\label{main_theorem}
\end{theorem}
\begin{remark}
From this last estimate, we get the following inequality for all $1<p \leq \frac{3}{2}$
\begin{equation*}
    ||u_\lambda - \bar{u}||_{\dot{W}^{1,p}(\R^3)} \leq C \lambda^{\frac{5}{3}},
\end{equation*}
where $\bar{u}$ is solution to the effective problem \eqref{system::intro_effective_model}. Through Sobolev embedding, this yields 
\begin{equation*}
     ||u_\lambda - \bar{u}||_{L^r_{loc}(\R^3)} \leq C \lambda^{\frac{5}{3}} \quad \text{for} \quad r=\frac{3p}{3-p} \in \Big[\frac{3}{2},6 \Big].
\end{equation*}
\end{remark}\medskip

\noindent \textbf{Paper organization.}  The outline of the paper is as follows. In Section \ref{section::splitting}, we split the microscopic system \eqref{system::main} into two systems: the \textit{"active system"} which accounts for particles self-propulsion and the \textit{"passive system"} which characterizes the effects induced by immobile rigid particles. This splitting will be of major interest as the literature about the \textit{passive system} is already well-furnished. Subsection \ref{subsection::aproximation u_N^app} introduces an approximation of the \textit{active system} in the dilute regime. The analysis of this approximation performed in Subsections \ref{subsection approx} and \ref{subsection::study of v_N} is the key argument used to prove Theorem \ref{thm::u_N_bornée} in Section \ref{section::boundeness u_N}. The proof of Theorem \ref{main_theorem} is then carried out in Section \ref{section::proof of theorem 1}. Section \ref{section::comments} provides comments on the effective model as well as a discussion on the particles orientation distribution. Eventually, we explain in Section \ref{section::Relaxed Assumption} how one may relax the separation assumption \eqref{assumption::separation} through slight modifications of the analysis.

\section{Splitting of the solution $\boldsymbol{u_N}$}
\label{section::splitting}
Thanks to Stokes equations linearity, we split the solution of the microscopic problem \eqref{system::main} into two fields $u_N := u_N^{\frak{p}} + u_N^{\frak{a}}$. The first term $u_N^{\frak{p}}$ accounts for the particles presence as \textit{passive rigid particles} and is a solution of the \textit{passive system}
\begin{equation}
    \left\{
      \begin{aligned}
         - \mu \Delta u_N^{\frak{p}} + \na p_N^{\frak{p}} & = g_N  &&\ \text{in} \ \Omega_N,\\
         \div  u_N^{\frak{p}} &  = 0 &&\ \text{in} \ \Omega_N,\\
         u_N^{\frak{p}}(x) & = U_i^{\frak{p}} +V_i^{\frak{p}} \times (x-x_i) &&\text{ in } B_i,\\
         \int_{\partial B_i} \sigma(u_N^{\frak{p}},p_N^{\frak{p}})n & = - \int_{B_i} g_N &&\text{ for all } i,\\
         \int_{\partial B_i} \sigma(u_N^{\frak{p}},p_N^{\frak{p}})n \times (x-x_i) & = - \int_{B_i} g_N \times (x-x_i)  &&\text{ for all } i,\\
         \underset{|x| \rightarrow + \infty}{ \lim} u_N^{\frak{p}}(x) &  = 0.
      \end{aligned}
    \right.
    \label{system::main_passive}
\end{equation}
This system and similar ones were already of interest in several studies, see \cite{almog1999ensemble,sanchez1985einstein,Haines&Mazzucato,HillairetWu,niethammer2020local,duerinckx2020effective,DuerinckxGloria,GVRH,gerard2020correction,gerard2021derivation} among many. Through a variational argument and an energy estimate, one can easily show
\begin{proposition}
\label{prop::well_posedness_passive}
For any $N>0$, there exists a unique solution $u_N^{\frak{p}}$ in $\dot{H}^1(\R^3)$ to the system \eqref{system::main_passive} such that $|| \na u_N^{\frak{p}}||_{L^2(\R^3)} \leq C ||g_N||_{L^{\frac{6}{5}}(\R^3)}$ for a constant $C$ independent of $N$.
\end{proposition}
More precisely, an order-two correction to the famous Einstein formula for effective viscosity is derived in \cite{gerard2020correction} from this system. Adapting the arguments to the simpler order-one case, we have 
\begin{theorem}[adapted from \cite{gerard2020correction}]
\label{thm::DGV_Mecherbet}
Let $\lambda>0$. Assume \eqref{assumption::empirical_measure_spatial}-\eqref{assumption::separation} hold and that $g_N \rightarrow g$ in $L^{\frac{6}{5}}(\R^3) \cap L^{3+\eps}(\R^3)$, for some $\eps >0$. Any accumulation point $u_\lambda^{\frak{p}}$ of $(u_N^{\frak{p}})_N$ in $\dot{H}^1(\R^3)$ is a solution in the sense of distributions to  
\begin{equation}
    \left\{
      \begin{aligned}
         - 2 \mu \, \div ([1+\frac{5}{2}\rho \lambda] D(u_\lambda^{\frak{p}})) + \na p_\lambda^{\frak{p}} & = g + \mR^{\frak{p}}_\lambda &&\ \text{in} \ \R^3,\\
         \div  u_\lambda^{\frak{p}} &  = 0 &&\ \text{in} \ \R^3,\\
         \underset{|x| \rightarrow + \infty}{ \lim} u_\lambda^{\frak{p}}(x)  & = 0
      \end{aligned}
    \right.
    \label{system::effective_passive}
\end{equation}
where $\mR^{\frak{p}}_\lambda$ is a remainder term that satisfies for all $q \geq 3$ 
\begin{align*}
    & \langle  \mR^{\frak{p}}_\lambda | \phi \rangle \leq C \lambda^{\frac{5}{3}} ||D(\phi)||_{L^q(\R^3)}, \quad \forall \phi \in \dot{H}^1(\R^3) \cap \dot{W}^{1,q}(\R^3) \text{ divergence-free.}
\end{align*}
\end{theorem}\medskip

The second part $u_N^{\frak{a}}$ represents \textit{particles activity} and is a solution of the \textit{active system}
\begin{equation}
    \left\{
      \begin{aligned}
         - \mu \Delta u_N^{\frak{a}} + \na p_N^{\frak{a}} & = - k_f \sum_{i=1}^N \delta(x - x_{f,i})  \, \rp_i   &&\ \text{in} \ \Omega_N,\\
         \div u_N^{\frak{a}} &  = 0 &&\ \text{in} \ \Omega_N,\\
         u_N^{\frak{a}}(x) & = U_i^{\frak{a}} +V_i^{\frak{a}} \times (x-x_i) &&\text{ in } B_i,\\
         \underset{|x| \rightarrow + \infty}{ \lim} u_N^{\frak{a}}(x) &  = 0
      \end{aligned}
    \right.
    \label{system::main_active}
\end{equation}
endowed with the balance of forces and torques
\begin{align*}
     \int_{\partial B_i} \sigma(u_N^{\frak{a}},p_N^{\frak{a}})n  = - k_f \rp_i \quad \text{and} \quad
         \int_{\partial B_i} \sigma(u_N^{\frak{a}},p_N^{\frak{a}})n \times (x-x_i)  = 0, \quad i =1 \dots N.
\end{align*}
This problem has a solution in $\dot{W}_{loc}^{1,p}(\R^3)$ for $1<p<\frac{3}{2}$, obtained by splitting the above system in several sub-systems. A part of the work will be to explain how the above active system is responsible for the apparition of the so called \textit{active stress tensor} $\sigma_1$ in the effective fluid model \eqref{system::intro_effective_model}. In order to prove Theorem \ref{thm::u_N_bornée}, we need a splitting of the above system \eqref{system::main_active}. This is the purpose of the next Subsection.

\subsection{Approximation of $\boldsymbol{u_N^{\frak{a}}}$}
\label{subsection::aproximation u_N^app}
In order to deal with the irregular source terms in \eqref{system::main_active}, we approximate $u_N^{\frak{a}}$ by a dilute approximation denoted $u_N^{app}$. Namely, $u_N^{app}$ should deal with the irregular Dirac source terms as well as the self-propulsion on each particle. As we work in a dilute regime, we neglect all inter-particle interactions and $u_N^{app}$ is defined as a sum of elementary solutions for each particle
\begin{equation}
    (u^{app}_N,p^{app}_N):= \sum_{i=1}^N (v[\rp_i], p[\rp_i]) (x-x_i),
    \label{def::u_N^app_regime_dilué}
\end{equation}
where $v[\rp]$ denotes the solution of the following elementary problem outside the ball $B:= \overline{B(0,a)}$
 \begin{equation}
    \left\{
      \begin{aligned}
        - \mu \Delta v[\rp] + \na p [\rp] & = - k_f \delta(x - a\beta \rp)  \, \rp  &&\ \text{in} \ \R^3 \setminus B,\\
        \div v[\rp] & = 0 &&\ \text{in} \ \R^3 \setminus B,\\
       v[\rp](x) & = U + V \times x &&\text{ in } B,\\
        \underset{|x| \rightarrow + \infty}{ \lim} v[\rp](x) &  = 0,
      \end{aligned}
    \right.
    \label{system::elementary_active_problem}
\end{equation}
endowed with the force and torque balance
\begin{equation*}
    \int_{\partial B}\sigma(v[\rp],p[\rp]) n + k_f \rp =0 \quad \text{and} \quad \int_{ \partial B} \sigma(v[\rp],p[\rp]) n \times x = 0.
\end{equation*}
Note that construction of type \eqref{def::u_N^app_regime_dilué} corresponds to the first iteration of the well-known \textit{Method of Reflections}, see \cite{hofer2021convergence} for an in-depth review of the method. In Subsection \ref{subsection approx}, we first explicitly construct the elementary solution $v[\rp]$, see Lemma \ref{lemma::construction de v[p]}, and we then show that $u_N^{app}$ has a limit in $L_{loc}^p(\R^3)$, $1 < p < \frac{3}{2}$, in Proposition \ref{convergence u_app^N}. \medskip

 We denote $v_N$ the associated error field $v_N:= u_N^{\frak{a}}-u^{app}_N$, which is the solution of the following microscopic system 
\begin{equation}
    \left\{
      \begin{aligned}
         - \mu \Delta v_N + \na \pi_N & = 0  &&\ \text{in} \ \Omega_N,\\
         \div v_N & = 0 &&\ \text{in} \ \Omega_N,\\
         v_N(x) & = - h_i(x) + \tilde{U}_i + \tilde{V}_i \times (x-x_i) &&\text{ in } B_i,\\
         \underset{|x| \rightarrow + \infty}{ \lim} v_N(x) &  = 0
      \end{aligned}
    \right.
    \label{systeme::error_v_N}
\end{equation}
where $h_i$ is the error realized by the dilute approximation $u_N^{app}$ in particle $B_i$ 
\begin{align*}
    h_i(x) = \sum_{j \neq i} v[\rp_j](x-x_j), \quad \text{for} \quad x \in B_i.
\end{align*}
System \eqref{systeme::error_v_N} is coupled with the balance equations for all particles
\begin{equation*}
    \int_{\partial B_i}\sigma(v_N,\pi_N){} n = 0 \quad \text{and} \quad \int_{ \partial B_i} \sigma(v_N,\pi_N) n \times (x-x_i)  = 0.
\end{equation*}
The solution $v_N$ is the unique minimizer of
\begin{equation*}
    \left\{ \int_{\R^3} |D(u)|^2,\  u \in \dot{H}^1(\R^3), \ \div u =0, \text{ s.t } D(u)=-D(h_i) \text{ on } B_i \text{ for all } i=1 \dots N \right\}.
\end{equation*}
We will prove in Subsection \ref{subsection::study of v_N} that $v_N$ is uniformly bounded in $\dot{H}^1(\R^3)$, and therefore in $L^p_{loc}(\R^3)$ for all $1<p \leq 6$ and thus for all $1<p<\frac{3}{2}$.

\section{Bound of the solution - proof of Theorem \ref{thm::u_N_bornée}}
\label{section::boundeness u_N}
This section is devoted to the proof of Theorem \ref{thm::u_N_bornée}. According to  Proposition \ref{prop::well_posedness_passive}, we have a uniform control of $u_N^{\frak{p}}$ in $\dot{H}^{1}(\R^3)$ and thus in $L_{loc}^p(\R^3)$ for $1 \leq p \leq 6$. Using the splitting $u_N = u_N^{\frak{p}} + u_N^{app} + v_N$ we introduced in the previous Section, it remains to respectively control the approximation $u_N^{app}$ and the remainder $v_N$. The next Subsection demonstrates a convergence result for $u_N^{app}$ in $L^p_{loc}(\R^3)$, $1 < p < 3/2$, and Subsection \ref{subsection::study of v_N}  provides a uniform estimate of $v_N$ in $\dot{H}^{1}(\R^3)$, along with one additional result. Gathering these results together, the Theorem follows.
\subsection{Convergence of the dilute approximation $\boldsymbol{u_{N}^{app}}$}
\label{subsection approx}
We focus here on the proof of the following convergence result 
\begin{proposition}
\label{convergence u_app^N}
The approximation $u^{app}_N$ defined in \eqref{def::u_N^app_regime_dilué} weakly converges in $L_{loc}^p(\R^3)$ for any $ 1 <p<\frac{3}{2}$ toward $w_0:=St^{-1}( \lambda \, \div \sigma_1)$ the solution of the Stokes equations
\begin{equation}
    \left\{
      \begin{aligned}
         - \mu \Delta w_0 + \na p_0 & =  \lambda \, \div \sigma_1 &&\text{ in }  \R^3,\\
         \div w_0 & = 0 &&\ \text{in} \ \R^3,\\
         \underset{|x| \rightarrow + \infty}{ \lim} w_0(x)  & = 0.
      \end{aligned}
    \right.
    \label{systeme::w_0}
\end{equation}
where 
\begin{align}
    \sigma_1(x):= \alpha \mJ  \int_{\mS^2} \rp \otimes \rp f(x,\rp) \, \d \rp,
    \label{def::sigma_1}
\end{align}
and $\mJ$ is the positive number given by $\mJ := \frac{3 \alpha \mu}{4}\Big(\beta- \frac{5}{2} \beta^{-2} + \frac{3}{2} \beta^{-4} \Big) .$
\end{proposition}
\begin{remark}
\label{remark::sigma_1 divergence_free}
As this last system is tested against divergence-free test function, we may write without any changes 
\begin{equation*}
     \sigma_1(x) = \alpha \mJ \int_{\mS^2} (\rp \otimes \rp -  \frac{1}{3}Id)f(x,\rp) \, \d \rp
\end{equation*}
in order to have a trace-free matrix $\sigma_1$, that matches the announced formula \eqref{first def:: sigma_1}. This expression corresponds to an orientation average of the \textit{dipole moment} $\rp \otimes \rp - \frac{1}{3} Id$. This dipole representation was commonly used in several studies when computing the disturbance flow from an active particle, see \cite{potomkin2016effective} for instance.
\end{remark}
In order to prove this Proposition, we first construct an explicit solution $v[\rp]$ of  \eqref{system::elementary_active_problem} and we collect several useful properties  \medskip

\noindent $\bullet$ \textit{Explicit solution} \medskip

\noindent We introduce the Oseen tensor $U(x):= \frac{1}{8 \pi \mu} \Big(\frac{Id}{|x|}+ \frac{x \otimes x}{|x|^3} \Big)$ as well as related ones defined by the explicit formulas
\begin{align*}
     \Tilde{U}(x):= \frac{1}{8 \pi \mu} \Big( \frac{Id}{3 |x|} - \frac{x \otimes x}{|x|^3}\Big) \ \quad \text{and} \quad \ \Delta U(x):=\frac{1}{8 \pi \mu } \Big( \frac{2 Id}{|x|^3} - \frac{6}{|x|^5} x \otimes x \Big), \quad x \in \R^3 \setminus \{0 \}.
\end{align*}
The third-rank tensor $\nabla U$ is defined for any matrix $A$ by
\begin{equation*}
    \nabla U(x)A = -\frac{3}{8 \pi \mu} \frac{A:x \otimes x}{|x|^5} x, \quad x \in \R^3 \setminus \{0 \}.
\end{equation*}
Using these fundamental quantities, we have
    \begin{lemma}
\label{lemma::construction de v[p]}
 A solution $v[\rp]$ of \eqref{system::elementary_active_problem} is given by 
 \begin{align}
     & v[\rp](x) \nonumber\\
      = & -k_f U(x-a\beta \rp) \, \rp + k_f \gamma_1 U(x-a\beta^{-1} \rp) \, \rp + k_f \gamma_2 a \na U(x-a\beta^{-1} \rp) \big(\rp \otimes \rp- \frac{Id}{3} \big) \label{expression v_elem}\\
    & + k_f \gamma_3 a^2\Delta U(x-a\beta^{-1} \rp) \, \rp + k_f (1-\gamma_1) U(x) \, \rp + k_f(1-\gamma_1) \frac{a^2}{|x|^2} \Tilde{U}(x) \, \rp, \quad x \in \R^3 \setminus B \nonumber
 \end{align}
where we denoted 
\begin{align*}
     \gamma_1:= \frac{3 \beta^{-1}}{2}  - \frac{\beta^{-3} }{2}, \quad \gamma_2:=  \beta^{-2} - \beta^{-4} \quad \text{and} \quad \quad \gamma_3:= \frac{\beta^{-1}}{4}(1-\beta^{-2})^2
\end{align*}
constants only depending on the parameter $\beta$. \medskip
 
  \noindent Extending this field by $\frac{k_f (1-\gamma_1)}{6 \pi \mu a} \rp $ inside the ball $B$, the total field $v[\rp]$ lies in $\dot{W}_{loc}^{1,r}(\R^3)$, $ 1 \leq r < 3$. 
\end{lemma}
\noindent The proof is given in  Appendix \ref{appendix::lemma_v[p]} and the associated pressure $p[\rp]$ is computed there for the sake of completeness. \medskip

\noindent Furthermore, $v[\rp]$ satisfies in the sense of distributions 
\begin{equation}
    \left\{
      \begin{aligned}
 - \mu \Delta v[\rp] + \nabla p[\rp] & = - k_f \delta(x-a \beta \rp) \, \rp - \sigma(v[\rp],p[\rp])^+n \, s^a &&\text{ in } \R^3,      \\
     \div v[\rp] & = 0 &&\text{ in } \R^3,
      \end{aligned}
    \right.
\label{edp::solution_elementaire_v[p]}
\end{equation}
where $s^a$ denotes the surface measure on $\partial B(0,a)$ and $\sigma(v[\rp],p[\rp])^+$ is the value of the constraint tensor coming from the exterior of the sphere. \medskip

\noindent $\bullet$ \textit{Scaling consideration} \medskip

\noindent  We have
\begin{equation}
    v[\rp](ax)= \frac{k_f}{a} w[\rp](x), \quad x \in \R^3
    \label{scaling_v[p]}
\end{equation}
where
\begin{align*}
      w[\rp](x) :=&   - U(x-\beta \rp) \, \rp + \gamma_1 U(x-\beta^{-1} \rp) \, \rp +  \gamma_2 \na U(x-\beta^{-1} \rp) \big(\rp \otimes \rp- \frac{Id}{3}\big)\\
    & +  \gamma_3 \Delta U(x-\beta^{-1} \rp) \, \rp + (\gamma_1-1) U(x) \, \rp + (\gamma_1-1) \frac{1}{|x|^2} \Tilde{U}(x) \, \rp, \quad x \in \R^3 \setminus B(0,1)\\
     w[\rp](x)   :=&  \frac{1-\gamma_1}{6 \pi \mu} \rp, \quad x \in B(0,1).
\end{align*}
This field $w[\rp]$ defines a divergence-free field in $\dot{W}^{1,r}_{loc}(\R^3)$, $1 \leq r < 3$, and which is independent of $a$. Denoting $p_w(x) = \frac{a}{k_f} p[\rp](x)$ the pressure associated to this scaling, we verify
\begin{equation*}
    \left\{
      \begin{aligned}
        - \mu \Delta w[\rp] + \na p_w [\rp] & = - \delta(x - \beta \rp)  \, \rp  &&\ \text{in} \ \R^3 \setminus B(0,1),\\
        \div w[\rp] & = 0 &&\ \text{in} \ \R^3 \setminus B(0,1),\\
       w[\rp](x) & = U + V \times x &&\text{ in } B(0,1),\\
        \underset{|x| \rightarrow + \infty}{ \lim} w[\rp](x) &  = 0,
      \end{aligned}
    \right.
\end{equation*}
along with 
\begin{equation*}
    \int_{\partial B}\sigma(w[\rp],p_w[\rp]) n + \rp =0 \quad \text{and} \quad \int_{ \partial B} \sigma(w[\rp],p_w[\rp]) n \times x = 0.
\end{equation*}

\medskip

\noindent $\bullet$ \textit{Reduction to a Stokeslet}\medskip

\noindent Through Taylor expansions with integral remainder, the expression \eqref{expression v_elem} of the elementary solution $v[\rp]$ simplifies to a \textit{Stokeslet doublet} $\nabla U$. Indeed, using that $k_f=\alpha \pi \mu a^2$ and $a^3 = \frac{3\lambda  }{4 \pi N}$, we write
\begin{align}
    v[\rp](x)& = \frac{\lambda}{N} \mu \na U(x) C(\rp)+ R[\rp](x), \quad x \in \R^3 \setminus B     \label{taylor_expansion_v[p]}\\
    \text{ where }C(\rp)  &:=  \alpha \mJ \rp \otimes \rp - \mJ' Id, \quad \mJ' \quad  \text{is a real number,}  \nonumber 
\end{align}
with the following remainder
\begin{align*}
R[\rp](x) = &  -k_f a^2 \beta^2 \int_0^1 (1-t) \nabla^2 U(x-t a \beta \rp) \rp \otimes \rp  \, \rp \, \d t \\
& +  k_f \gamma_1 a^2 \beta^{-2} \int_0^1 (1-t) \na^2 U(x-a t \beta^{-1} \rp)  \rp \otimes \rp  \, \rp \, \d t\\
& -  k_f \gamma_2 a^2 \beta^{-1}  \int_0^1 \na^2 U(x-a t \beta^{-1} \rp)  \big( \rp \otimes \rp - \frac{Id}{3} \big) \, \rp \, \d t \\
& + k_f \gamma_3 a^2\Delta U(x-a\beta^{-1} \rp) \, \rp \\
& + k_f (1-\gamma_1) \frac{a^2}{|x|^2} \Tilde{U}(x) \, \rp.
\end{align*}
Since all the operators appearing in the remainder are homogeneous of degree $-3$, we easily verify
\begin{equation}
    |R[\rp](x)| = \mathcal{O}\big( a^4 (|x|^{-3} +|x-a \beta \rp|^{-3} + |x-a \beta^{-1} \rp|^{-3}) \big).
    \label{remainder_grand_O_v[p]}
\end{equation}
 Note that it was expected that the main contribution of this elementary solution would be a \textit{Stokeslet doublet} $\nabla U$ since the solution expression \eqref{expression v_elem} displayed two pairs of opposed fundamental solutions at very closed points.  It is also known from the literature that an active particle disturbance flow may be seen as a  \textit{Stokeslet doublet} (or dipole) in the far field, see for instance \cite{drescher2011fluid} or \cite{lauga2016stresslets}. \medskip
 
 We are now ready to prove Proposition \ref{convergence u_app^N}.

\begin{proof}[Proof of Proposition \ref{convergence u_app^N}]
We claim first that $- \mu \Delta u^{app}_N + \na p^{app}_N$ converges in the sense of distributions toward $\lambda \, \div{} \sigma_1$. We know from \eqref{edp::solution_elementaire_v[p]} that $u^{app}_N = \sum_{j= 1}^N v[\rp_j](\cdot-x_j)$ solves in the sense of distributions
\begin{align*}
     - \mu \Delta u^{app}_N + \na p^{app}_N = -k_f \sum_{j=1}^N \delta(\cdot - x_{f,j}) \, \rp_j - \sum_{j=1}^N [(\sigma(v[\rp_j],p[\rp_j])^+ n  ](\cdot) s^a (\cdot + x_j) \quad \text{in} \quad \R^3 
\end{align*}
along with $ \div u_N^{app} = 0 \text{ in } \R^3$. Therefore, testing the above equation against $\varphi$ in $D(\R^3)$, we find through a change of variables
\begin{align*}
    & \langle - \mu \Delta u^{app}_N + \na p^{app}_N | \varphi \rangle \\
    & = -k_f \sum_{j} \varphi(x_{f,j}) \cdot \rp_j - \sum_{j} \int_{\partial B(0,a)} [\sigma(v[\rp_j],p[\rp_j])^+ n](x)\cdot \varphi(x+x_j) \, \d s (x)\\
    & = -k_f \sum_{j} \varphi(x_{f,j}) \cdot \rp_j - \sum_{j} \int_{\partial B(x_j,a)} [\sigma(v[\rp_j],p[\rp_j])^+ n](x-x_j)\cdot \varphi(x) \, \d s (x)
\end{align*}
We approximate $\varphi(x)$ by its value at the particles centers to write
\begin{align*}
    & \sum_{j} \int_{\partial B(x_j,a)} [ \sigma(v[\rp_j],p[\rp_j])^+ n](x-x_j) \cdot \varphi(x) \, \d s (x) \\
    & = \sum_{j} \int_{\partial B(x_j,a)} [\sigma(v[\rp_j],p[\rp_j])^+ n](x-x_j) \, \d s(x)  \cdot \varphi(x_j)\\
    & + \sum_{j} \int_{\partial B(x_j,a)} [\sigma(v[\rp_j],p[\rp_j])^+ n](x-x_j) \cdot (\varphi(x)-\varphi(x_j)) \, \d s (x).
\end{align*}
Using the force balance on particle $B_j$ we have 
\begin{equation*}
    \int_{\partial B(x_j,a)} [\sigma(v[\rp_j],p[\rp_j])^+ n](x-x_j) \, \d s(x) \cdot \varphi(x_j) = - k_f  \rp_j \cdot \varphi(x_j),
\end{equation*}
thus we may write
\begin{align}
   &\langle - \mu \Delta u^{app}_N + \na p^{app}_N | \varphi \rangle  = k_f \sum_{j} (\varphi(x_j)- \varphi(x_{f,j})) \cdot \rp_j + \langle B_N | \varphi \rangle,    \label{stokes_u_N^app}\\
& \text{where } \langle B_N | \varphi \rangle := \sum_{j} \int_{\partial B(x_j,a)} [\sigma(v[\rp_j](\cdot-x_j),p[\rp_j](\cdot-x_j))^+ n](x) \cdot (\varphi(x)-\varphi(x_j)) \, \d s(x) .\nonumber
\end{align}
 To compute the first term's limit, we write using \eqref{assumption::force}:
\begin{align}
 &  k_f \sum_{j} (\varphi(x_j)- \varphi(x_{f,j})) \cdot \rp_j = -\frac{\lambda}{N} \frac{3\alpha  \mu \beta}{4 } \sum_{j} \nabla \varphi(x_j) : \rp_j \otimes \rp_j + \mR_N  \label{equality::convergence_utile 3_8}\\
&\text{ with } \mR_N:=  \frac{\lambda}{a N}\frac{3\alpha \mu  }{4} \sum_{j} \left[ (\varphi(x_j)- \varphi(x_{f,j})) \cdot \rp_j + a \beta \nabla \varphi(x_j) : \rp_j \otimes \rp_j \right] \nonumber.
\end{align}
Thanks to a Taylor expansion we get, reminding that $x_{f,j}=x_j + a \beta \rp_j$:
\begin{align*}
    | (\varphi(x_j)- \varphi(x_{f,j})) \cdot \rp_j + a \beta \nabla \varphi(x_j) : \rp_j \otimes \rp_j | \leq C a^2||\na^2 \varphi||_{L^\infty(\R^3)}, \quad  j=1 \dots N
\end{align*}
with a constant $C$ independant of $j= 1 \dots N$. Hence, we find
\begin{align*}
     |\mR_N| \leq  \frac{a}{N} \sum_{j} C \leq C a 
\end{align*}
which vanishes as $N$ increases since $a \sim N^{-1/3}$. Consequently, we get, using the empirical measure convergence \eqref{assumption::empirical_measure} in \eqref{equality::convergence_utile 3_8}
\begin{align}
    k_f \sum_{j} (\varphi(x_j)- \varphi(x_{f,j})) \cdot \rp_j \underset{N \rightarrow + \infty}{\longrightarrow} -\lambda \frac{3\alpha  \mu \beta    }{4} \int_{\R^3 \times \mS^2} \na \varphi(x) : \rp \otimes \rp f(x,\rp) \, \d x \d \rp.
    \label{limite term A_N}
\end{align}
We now turn to $\langle B_N | \varphi \rangle$. Using the scaling consideration \eqref{scaling_v[p]}, we obtain through the change of variables $x=ay+ x_j$ 
\begin{align*}
    \langle B_N | \varphi \rangle & =\sum_{j} \int_{\partial B(x_j,a)} [\sigma(v[\rp_j],p[\rp_j])^+n](x-x_j) \cdot (\varphi(x)-\varphi(x_j)) \, \d s(x)\\
    & = a^2 \sum_{j} \int_{\partial B(0,1)} [\sigma(v[\rp_j],p[\rp_j])^+ n]( a y) \cdot (\varphi(ay + x_j)-\varphi(x_j)) \, \d s(y)\\
    & =  k_f  \sum_{j} \int_{\partial B(0,1)} [\sigma(w[\rp_j],p_w[\rp_j])^+ n](y)\cdot (\varphi(a y + x_j)-\varphi(x_j)) \, \d s(y).
\end{align*}
Owing to a similar Taylor expansion, we write
\begin{equation*}
    \langle B_N | \varphi \rangle = a k_f \sum_{j} \int_{\partial B(0,1) } [\sigma(w[\rp_j],p_w[\rp_j])^+ n](y) \otimes y \, \d s(y) :\na \varphi(x_j)+ \mR_N'
\end{equation*}
where $\mR_N$ is a remainder satisfying $|\mR_N'| \leq k_f a^2 \sum_j ||\na^2 \varphi||_{L^\infty(\R^3)} = \mathcal{O}(a)$. To compute the limit of the sum involving $\sigma(w[\rp_j],p_w[\rp_j])$, remark that the quantity 
$$\int_{\partial B(0,1) } [\sigma(w[\rp_j],w[\rp_j])^+ n](y) \otimes y \, \d s (y)$$
is a physical quantity called \textit{Stresslet} in the literature, which corresponds to the symmetric part of the first force moment on particle $B(0,1)$. Using tedious computations from \cite[Section 10.2.1]{KK} coupled with the fact that a translating sphere in a Stokes flow produces no Stresslet, we get\footnote{$w[\rp]$ must be splitted like $v[\rp]$ to compute the \textit{Stresslet}, see the proof of Lemma \ref{lemma::construction de v[p]} for details.} for all $j=1 \dots N$:
\begin{align*}
    \int_{\partial B(0,1) } [\sigma(w[\rp_j],p_w[\rp_j])^+ n](y) \otimes y  \, \d s (y) = \Big( - \frac{5}{2} \beta^{-2} + \frac{3}{2} \beta^{-4} \Big) \, \rp_j \otimes \rp_j.
\end{align*}
Passing to the limit, we find 
\begin{align}
    \underset{N \rightarrow + \infty}{\lim} \, \langle B_N | \varphi \rangle =  \lambda\frac{3 \alpha \mu}{4} \Big( - \frac{5}{2} \beta^{-2} + \frac{3}{2} \beta^{-4} \Big)  \int_{\R^3 \times \mS^2} \rp \otimes \rp : \na \varphi(x) f(x,\rp) \, \d x \d \rp.
    \label{limite_term_B_N}
\end{align}
Combining \eqref{limite term A_N}-\eqref{limite_term_B_N} with a last integration by parts, it follows that $\langle - \mu \Delta u^{app}_N + \na p^{app}_N | \varphi \rangle$ converges toward
\begin{align*}
     -\lambda \frac{3\alpha  \mu   }{4} \Big(\underbrace{\beta  - \frac{5}{2} \beta^{-2} + \frac{3}{2} \beta^{-4} }_{=\mJ}\Big) \int_{\R^3} \Big( \int_{\mS^2} \rp \otimes \rp f(x,\rp) \d \rp \Big): \na \varphi(x) \, \d x= \langle  \lambda \, \div \sigma_1 | \varphi \rangle,
\end{align*}
which proves that $- \mu \Delta u^{app}_N + \nabla p^{app}_N$ converges in the sense of distributions toward $\lambda \, \div{} \sigma_1$ with the expected matrix $\sigma_1$ defined in \eqref{def::sigma_1}. \medskip

 By a duality argument, we will show how this convergence implies that $u^{app}_N$ weakly converges toward $St^{-1}(\lambda \, \div{} \sigma_1)$ in $L^p_{loc}(\R^3)$ for $1<p<\frac{3}{2}$. We will need the following inequality 
\begin{equation}
     |\langle - \mu \Delta u^{app}_N + \nabla p^{app}_N | \varphi \rangle |  \leq C ||\na \varphi||_{L^\infty(\R^3)} \label{inequality_stokes_u_N^app}
\end{equation}
that we easily obtain  from \eqref{stokes_u_N^app} for any test function in $D(\R^3)$. Let $K$ be a compact set and $\Psi$ in $L^{p'}(K)$ for $p'>3$. We extend $\Psi$ by zero outside $K$ and we note $\tilde{\Psi}$ that extension. Let $\varphi$ be the solution of $St(\varphi) = \tilde{\Psi}$. From elliptic regularity, we know that $\varphi$ is in $W^{2,p'}(\R^3)$ and a Sobolev embedding gives that $\na \varphi \in L^\infty(\R^3)$. We compute using an integration by parts
\begin{align*}
    \int_{K} u_N^{app} \cdot \Psi & = \int_{\R^3}u_N^{app} \cdot \tilde{\Psi} = \int_{\R^3} u_N^{app} \cdot St(\varphi) = \int_{\R^3} St(u_N^{app}) \cdot \varphi.
\end{align*}
Owing to a density argument, the inequality \eqref{inequality_stokes_u_N^app} remains true for $\varphi$ in  $W^{2,p'}(\R^3)$ and yields
\begin{align*}
    \Big| \int_{K} u_N^{app} \cdot \Psi \Big| = \Big| \int_{\R^3} St(u_N^{app}) \cdot \varphi \Big|  \leq  C ||\na \varphi||_{L^\infty(\R^3)}.
\end{align*}
The function $\varphi$ is explicit and given by $\varphi = U \star \tilde{\Psi}$ where $U$ is the Oseen tensor. For almost any $x$ in $\R^3$, we have 
\begin{align*}
    \na \varphi(x) & = \int_{\R^3} \na U(x-y) \tilde{\Psi}(y) \, \d y = \int_{K} \na U(x-y) \Psi(y) \, \d y
\end{align*}
and taking advantage of the fact that $\na U$ is homogeneous of degree $-2$, we get
\begin{align*}
    |\na \varphi(x)| & \leq \int_{K} \frac{1}{|x-y|^2} |\Psi(y)| \, \d y \leq \Big( \int_{K} \frac{1}{|x-y|^{2p}} \, \d y \Big)^{\frac{1}{p}} \Big( \int_{K} |\Psi|^{p'} \Big)^{\frac{1}{p'}} \leq C ||\Psi||_{L^{p'}(\R^3)}
\end{align*}
since $y \mapsto \frac{1}{|x-y|^{2p}}$ is locally integrable ($p'>3$). To sum up, we proved that for any compact $K$ and any $\Psi \in L^{p'}(K)$, the following bound holds
\begin{align*}
   \Big| \int_{\R^3} u_N^{app} \cdot \Psi \Big| \leq C ||\Psi||_{L^{p'}(K)}§.
\end{align*}
By duality, this implies that $u^{app}_N$ is uniformly bounded in $L^p(K)$ and thus there exists a subsequence (still denoted $u^{app}_N$) of $u^{app}_N$ that weakly converges  toward $w_0$ in $L^p(K)$. \medskip

The final step of the proof consists in showing that $w_0$ is the expected limit $St^{-1} (\lambda \, \div \sigma_1)$. Since $u_N^{app}$ weakly converges toward $w_0$ in $L_{loc}^p(\R^3)$, we have \begin{align*}
    - \mu \Delta u^{app}_N  \underset{N \rightarrow + \infty}{ \longrightarrow} - \mu \Delta w_0 \quad \text{in the sense of distributions.}
\end{align*}
Testing against a divergence-free function $\varphi$, we can add the pressure and write 
\begin{align*}
     \langle - \mu \Delta u^{app}_N | \varphi \rangle    &  = \langle - \mu \Delta u^{app}_N + \na p^{app}_N| \varphi \rangle \underset{N \rightarrow + \infty}{\longrightarrow} \langle \lambda \, \div \sigma_1 | \varphi \rangle.
\end{align*}
By uniqueness of the limit in the sense of distributions, we get  $\langle - \mu \Delta w_0 | \varphi \rangle = \langle \lambda \, \div \sigma_1 | \varphi \rangle$  for all divergence-free $\varphi$ in  $D(\R^3)$, which proves that $w_0$ is the solution in the sense of distributions of \eqref{systeme::w_0}. Classical arguments show that system \eqref{systeme::w_0} has a unique solution, which proves that $u_N^{app}$ has only one accumulation point, given by $w_0$. Altogether, this proved that the whole sequence $u_N^{app}$ converges toward $w_0 = St^{-1}(\lambda \, \div \sigma_1)$ in $L_{loc}^p(\R^3)$ for $1<p<\frac{3}{2}$.

\end{proof}

\subsection{Results on the remainder $\boldsymbol{v_N}$}
\label{subsection::study of v_N}

\begin{proposition}
For all $q \geq 2$ and all divergence-free $\phi$ in $C^\infty_0(\R^3)$, 
\begin{align*}
    \underset{N \rightarrow + \infty}{\limsup} \, \Big| 2 \mu \int_{\R^3} D(v_N):D(\phi) \Big| \leq C \lambda^{\frac{1}{2}-\frac{1}{q}} ||D( v_N)||_{L^2(\R^3)} ||D(\phi)||_{L^q(\R^3)}.
\end{align*}
\label{prop:limsup v_N}
\end{proposition}
We will then show that 
\begin{proposition}
For all $N >0$, 
\begin{align*}
    ||D(v_N)||_{L^2(\R^3)} \leq C \lambda^{\frac{3}{2}}.
\end{align*}
\label{prop_borne_uniforme_v_N}
\end{proposition}
\noindent
Mixing these two results together for $q \geq 3$ yields the following useful inequality for any divergence-free $\phi$ in $C_0^{\infty}(\R^3)$:
\begin{align}
    \underset{N \rightarrow + \infty}{\limsup} \, \Big| 2 \mu \int_{\R^3} D(v_N):D(\phi) \Big| \leq C \lambda^{\frac{5}{3}} ||D(\phi)||_{L^q(\R^3)}.
    \label{inequality::nouvelle useful limsup}
\end{align}

\begin{proof}[Proof of Proposition \ref{prop:limsup v_N}]
Using an integration by parts on \eqref{systeme::error_v_N}, we write
\begin{align*}
    2 \mu \int_{\R^3} D(v_N):D(\phi) & = 2 \mu \int_{\Omega_N} D(v_N): D(\phi) + 2 \mu \sum_i \int_{B_i} D(h_i):D(\phi)  \\
    & =- \sum_i \int_{\partial B_i} \sigma(v_N,\pi_N)n \cdot \phi +2 \mu \sum_i \int_{B_i} D(h_i):D(\phi).
\end{align*}
We now use a classical Bogovskii argument. For all $i$ and any constant vectors $v_i$ and $w_i$, we have owing to the balance equations on $v_N$:
\begin{align*}
    \int_{\partial B_i} \sigma(v_N,\pi_N)n \cdot \phi = \int_{\partial B_i} \sigma(v_N,\pi_N)n \cdot (\phi+ u_i +  v_i \times (x-x_i)).
\end{align*}
Thanks to the Bogovskii operator, there exists $\tilde{\phi}_i \in H^1_0(B(x_i, 2a))$ such that
\begin{align*}
    \div \tilde{\phi}_i = 0 \quad \text{in} \quad B(x_i, 2 a), \quad \tilde{\phi}_i = \phi + u_i + v_i \times (x-x_i) \quad \text{in} \quad  B_i
\end{align*}
along with $|| \na \tilde{\phi}_i||_{L^2(\R^3)} \leq C_{i,N} ||\phi + u_i + v_i \times (x-x_i) ||_{W^{1,2}(B_i)}$. With a good choices of the constant vectors $(u_i,v_i)_{1 \leq i \leq N}$, we have the Korn inequality: 
\begin{align*}
    ||\phi + v_i + w_i \times (x-x_i) ||_{W^{1,2}(B_i)} \leq C_{i,N} ||D( \phi)||_{L^2(B_i)},
\end{align*}
which leads to $|| \na \tilde{\phi}_i||_{L^2(\R^3)} \leq C_{i,N} ||D ( \phi)||_{L^2(B_i)}.$ With a translation and scaling argument, the family of constant $(C_{i,N})_i$ can be chosen independent of $i$ and $N$. Extending each $\tilde{\phi}_i$ by zero on the whole space, the sum $\tilde{\phi} = \sum_i \tilde{\phi}_i$ verifies
\begin{align*}
    ||\nabla \tilde{\phi}||_{L^2(\R^3)}^2 \leq C \sum_i ||D ( \phi)||_{L^2(B_i)}^2 = ||D ( \phi)||_{L^2( \cup_i B_i)}^2 \leq C || D(\phi)||_{L^2(\cup_i B_i)}^2
\end{align*}
as long as $ \underset{i \neq j}{\min} \, |x_i - x_j| > 2 a $ which is verified for small enough $\lambda$ thanks to \eqref{assumption::separation}. Using this extension $\tilde{\phi}$, we write
\begin{align*}
    2 \mu \int_{\R^3} D(v_N) : D(\phi) & = - \sum_i \int_{\partial B_i} \sigma(v_N,\pi_N)n \cdot \phi + 2 \mu \sum_i \int_{B_i} D(h_i):D(\phi)  \\
    & = - \sum \int_{\partial B_i} \sigma(v_N,\pi_N)n \cdot \tilde{\phi} + 2 \mu \sum_i \int_{B_i} D(h_i):D(\tilde{\phi}) \\
     & = 2 \mu \int_{\R^3} D(v_N) : D(\tilde{\phi}).
\end{align*}
The Cauchy-Schwarz and Hölder inequalities then yield the expected result
\begin{align*}
    \Big| 2 \mu  \int_{\R^3} D(v_N) : D(\tilde{\phi)}\Big| \leq C ||D(v_N)||_{L^2(\R^3)} ||D ( \phi)||_{L^2( \cup_i B_i)} \leq C \lambda^{\frac{1}{2}-\frac{1}{q}} ||D(v_N)||_{L^2(\R^3)} ||D ( \phi)||_{L^q(\R^3)}.
\end{align*}
\end{proof}

\begin{proof}[Proof of Proposition \ref{prop_borne_uniforme_v_N}]
From a variational argument borrowed from  \cite[Proposition 2.1]{gerard2020correction}, the unique solution $v_N$ of \eqref{systeme::error_v_N} satisfies, under the separation assumption  \eqref{assumption::separation} 
\begin{align*}
    ||\na v_{N}||_{L^2(\R^3)}^2 = 2 || D(v_{N})||_{L^2(\R^3)}^2 \leq C \sum_{i=1}^N ||D(h_i)||_{L^2(B_i)}^2
\end{align*}
for a constant $C$ independant of $N$. From the definition \eqref{expression v_elem} of $v[\rp_j]$ on the whole space, together with Taylor expansion \eqref{taylor_expansion_v[p]}, we write
\begin{align*}
    h_i(x) = \frac{\lambda}{N} \mu \sum_{j \neq i} \na U (x-x_j) C(\rp_j) + \sum_{j \neq i } R[\rp_j](x-x_j), \quad x \in B_i,
\end{align*}
which leads to
\begin{align}
    D(h_i)(x) = \frac{\lambda}{N} \mu \sum_{j \neq i} D(\na U)(x-x_j) C(\rp_j) + \sum_{j \neq i} D(R[\rp_j])(x-x_j), \quad x \in B_i.
    \label{taylor expansion D(h_i)}
\end{align}
As in \cite{gerard2020analysis}, we use the following notations: for any matrix $A$ and $x$ in $\R^3 \setminus \{ 0\} $,
$$D(\nabla U)(x) A = \frac{3}{8\pi\mu }\M(x)A, \quad \M(x)A = -2 \frac{Ax \otimes x}{|x|^5} + 5 \frac{A:(x \otimes x)}{|x|^7} x \otimes x $$
which gives $ D(\na U)(x-x_j) C(\rp_j) = \frac{3}{8 \pi} \M(x-x_j) C(\rp_j)$. Hence, we need to study the two following quantities $\mS$ and $\mR$:
\begin{multline*}
       \sum_{i} ||D(h_i)||_{L^2(B_i)}^2 \\
       \leq C  \sum_i  \int_{B_i} \Big| a^3 \sum_{j \neq i} \M(x-x_j) C(\rp_j) \Big|^2  \d x +C \sum_i \int_{B_i} \Big| \sum_{j \neq i} D(R[\rp_j])(x-x_j)\Big|^2  \d x := \mS + \mR.
\end{multline*}
 In order to deal with the sum of remainders term $\mR$, we recall (see \eqref{remainder_grand_O_v[p]}) that 
$$|D(R[\rp_j])(x-x_j)| = \mathcal{O}(a^4 (|x-x_j|^{-4} + |x-x_j - a \beta \rp_j|^{-4} + |x-x_j-a \beta^{-1} \rp_j|^{-4})).$$ 
We use this to bound $\mR$:
\begin{align*}
       \mR & =\sum_i \int_{B_i} \Big|  \sum_{j \neq i} D(R[\rp_j])(x-x_j) \Big|^2 \d x \\ 
     & \leq  |B_1| \sum_i  \Big( \Big| \sum_{j \neq i} \sup_{z \in B_i} \frac{a^4}{|z-x_j|^4} \Big|^2  +  \Big| \sum_{j \neq i} \sup_{(z,\rp) \in B_i \times \mS^2} \frac{a^4}{|z-x_j- a \beta \rp|^4} \Big|^2\\
     & +\Big| \sum_{j \neq i} \sup_{(z,\rp) \in B_i \times \mS^2} \frac{a^4}{|z-x_j - a \beta^{-1} \rp|^4}\Big|^2 \Big) \\
     & := I_1 + I_2 + I_3.
\end{align*}
We start with the analysis of $I_1$. The following inequality is satisfied for all $i \neq j$ using the separation assumption \eqref{assumption::separation}
\begin{align*}
   \underset{z \in B_i}{\sup} |x_j-z|^{-4} & =   (|x_i-x_j| - a)^{-4}  \leq C |x_i-x_j|^{-4} \big(1-  \lambda^{\frac{1}{3}} c^{-1} \big(\frac{4 \pi}{3}\big)^{-\frac{1}{3}}\big)^{-4} \leq C' |x_i-x_j|^{-4}
\end{align*}
where we used that $\lambda^{\frac{1}{3}} c^{-1} \big(\frac{4 \pi}{3} \big)^{-\frac{1}{3}} < \frac{1}{2}$ for small enough $\lambda$. Denoting $y_i:=x_i N^{1/3}$, we use $ |y_i-y_j| \geq \frac{1}{2} (|y_i-y_j| + c ) $ to find
\begin{align*}
     I_1 & \leq a^{11}  \sum_i   \Big|  \sum_{j \neq i}  |x_i-x_j|^{-4} \Big|^2  \\
     & \leq  C a^{11}  N^{8/3} \sum_i    \Big| \sum_{j} (c+|y_i-y_j|)^{-4} \Big|^2 \\
    & \leq C \lambda^{\frac{11}{3}} \Big( \sup_{i} \sum_j \frac{1}{(c+|y_i-y_j|)^4} \Big)^2 
\end{align*}
where we used the generalized Young inequality:
\begin{equation}
    \forall q \geq 1, \quad  \sum_i \big( \sum_j |a_{ij} b_j| \big)^q \leq \max \Big( \sup_i \sum_j |a_{ij}|, \sup_j \sum_i |a_{ij}| \Big)^{q} \sum_i |b_i|^q.
    \label{inequality::generalized_Young}
\end{equation}
We conclude the estimate noticing\footnote{see \ref{appendix::lemma serie} for a proof of a similar inequality.} that $\sum_j \frac{1}{(c+|y_i-y_j|)^4}$ is uniformly bounded in $i$ by the finite quantity $\int_{\R^3} \frac{1}{(c+|y|)^4} \, \d y$ which entails $I_1 \leq C \lambda^{\frac{11}{3}}$. To provide an estimate on $I_2$ and $I_3$, we use similar arguments, with a little twist to deal with the translated denominators. Notice that, for all $i \neq j$,
\begin{align*}
   \underset{(z,\rp) \in B_i \times \mS^2}{\sup} |z-x_j - a \beta \rp|^{-4} & =  (|x_i-x_j| - \beta a )^{-4}\\
 & \leq |x_i-x_j|^{-4} \big(1-  \beta\lambda^{\frac{1}{3}} c^{-1} \big(\frac{4 \pi}{3}\big)^{-\frac{1}{3}}\big)^{-4}\\
 & \leq C' |x_i-x_j|^{-4}
\end{align*}
where we used $ \beta\lambda^{\frac{1}{3}} c^{-1} \big(\frac{4 \pi}{3}\big)^{-\frac{1}{3}} < \frac{1}{2} $ from the well-separated assumption \eqref{assumption::separation_beta}. Using the same techniques as for $I_1$, we obtain $I_2 \leq C \lambda^{\frac{11}{3}}$. To bound $I_3$, the method is the same, but this time we use the fact that $\beta^{-1} < 1$ and therefore that \begin{align*}
  \underset{(z,\rp) \in B_i \times \mS^2}{\sup} |z-x_j - a \beta^{-1} \rp|^{-4} & = C  (|x_i-x_j| - \beta^{-1} a )^{-4} \leq C'  |x_i-x_j|^{-4}
\end{align*}
which leads in a similar way to $I_3 \leq C \lambda^{\frac{11}{3}}$ and thus to $\mR \leq C \lambda^{\frac{11}{3}}$. 

We now turn to the estimate on $\mS$ and we will prove that 
\begin{align*}
     \sum_i  \int_{B_i} \Big| a^3 \sum_{j \neq i} \M(x-x_j) C(\rp_j) \Big|^2  \d x \leq C \lambda^3.
\end{align*}
Using straightforwardly Lemma A.2 of \cite{HillairetWu}, one could show that the above quantity is bounded by $\lambda^2$. To improve this bound, we follow and adapt a similar argument in \cite{hofer2021convergence} (see proof of Proposition 4.1). Let $d_N = \frac{c}{4} N^{-\frac{1}{3}} \leq \frac{1}{4} \, \underset{i \neq j}{\min} \, |x_i - x_j|$, we introduce $l$ the following function
\begin{align*}
    l:=  \sum_j C(\rp_j) |B(x_j,d_N)|^{-1} 1_{B(x_j,d_N)}.
\end{align*}
Let $T$ be a function such that $T(y)=x_j$ if $y \in B(x_j,d_N)$. For any $x \in B_i$, we write
\begin{align*}
    a^3 \sum_{j \neq i} \M(x-x_j) C(\rp_j) & = a^3 \int_{\R^3 \setminus B(x_i,2 d_N)} (\M(x-T(y)) - \M(x-y)) l(y) \, \d y\\
    & + a^3 \dashint_{B(x_i,d_N) } \int_{\R^3 \setminus B(x_i,2 d_N)}  (\M(x-y) - \M(z-y)) l(y) \, \d y \d z\\
    & + a^3 \dashint_{ B(x_i,d_N) } \int_{\R^3 \setminus B(x_i,2 d_N)}  \M(z-y) l(y) \, \d y \d z\\
    & := A^i_1(x) + A^i_2(x) + A^i_3.
\end{align*}
For all $x \in B_i$ and $y \in B(x_j,d_N)$ with $i \neq j$, we have using \eqref{assumption::separation} 
\begin{align*}
    |\M(x-x_j) - \M(x-y)| \leq C |y-x_j| \sup_{\bar{z} \in [x_j,y]} |\bar{z}-x|^{-4} \leq C d_N |x_i-x_j|^{-4},
\end{align*}
and we obtain using similar computations as the one used above for $\mR$, still denoting $y_i:= N^{\frac{1}{3}} x_i$
\begin{align*}
    \sum_i ||A_1^i||_{L^2(B_i)}^2 & \leq C a^6 \sum_i \int_{B_i} \Big| \sum_{j \neq i} d_N \frac{|C(\rp_j)|}{|x_i-x_j|^4} \Big|^2 \,  \d x\\
    & \leq C a^{9} N^{\frac{8}{3}} d_N^2 \sum_i \Big| \sum_j C (c+|y_i-y_j|)^{-4} |C(\rp_j)| \Big|^2\\
    & \leq \frac{\lambda^3}{N} \Big( \sup_{i} \sum_j (c+|y_i-y_j|)^{-4} \Big)^2 \sum_i |C(\rp_i)|^2\\
    & \leq C \frac{\lambda^3}{N}  \sum_i |C(\rp_i)|^2.
\end{align*}
Thanks to \eqref{assumption::empirical_measure}, we find $  \sum_i ||A_1^i||_{L^2(B_i)}^2  \leq C \lambda^{3}$. The sum $\sum_i ||A_2^i||_{L^2(B_i)}^2 $ is treated the same way using that for all $x \in B_i$, $y \in  B(x_j,d_N)$ and $z \in  B(x_i,2 d_N)$, we have
\begin{align*}
    |\M(x-y) - \M(z-y)| \leq C d_N \sup_{\bar{z} \in [x,z]} |\bar{z} - y|^{-4} \leq C d_N (|x_i-x_j|- 2 d_N)^{-4} \leq C d_N |x_i-x_j|^{-4}
\end{align*}
and the same analysis yields $\sum_i ||A_2^i||_{L^2(B_i)}^2  \leq C \lambda^{3}$. We now turn to the contribution ofrom the $A_3^i$. Introducing the function $l_i$ defined by 
\begin{align*}
    l_i:=C(\rp_i) |B(x_i,d_N)|^{-1} 1_{B(x_i,d_N)}, \quad i =1 \dots N,
\end{align*}
we can write for any $z \in B(x_i, d_N)$:
\begin{align*}
    \int_{\R^3 \setminus B(x_i,2 d_N)}  \M(z-y) l(y) \, \d y = (\M \star (l-l_i))(z).
\end{align*}
We then deduce
\begin{align*}
    \sum_{i}||A_3^i||_{L^2(B_i)}^2 & = C a^3 \sum_{i} (A_3^i)^2\\
    & \leq C a^3 \sum_i \big( a^3 \dashint_{B(x_i,d_N)}  (\M \star (l-l_i))(z) \d z  \big)^2\\
    & \leq C  \frac{a^9}{d_N^3} \sum_i || \M  \star (l-l_i)||_{L^2(B(x_i,d_N))}^2 \quad \text{applying Hölder inequality}\\
    & \leq  C  \frac{a^9}{d_N^3} || \M  \star l||_{L^2(\R^3)}^2 + C  \frac{a^9}{d_N^3} \sum_i || \M  \star l_i||_{L^2(B(x_i,d_N))}^2
\end{align*}
using that the union $\cup_i B(x_i,d_N)$ is disjoint thanks to \eqref{assumption::separation}. As $\M$ is a Calderon-Zyngmund operator, we have 
\begin{align*}
     \sum_{i}||A_3^i||_{L^2(B_i)}^2  & \leq C  \frac{a^9}{d_N^3} ||l||_{L^2(\R^3)}^2 + C  \frac{a^9}{d_N^3} \sum_i ||l_i||_{L^2(\R^3)}^2.
\end{align*}
We then notice that $ \sum_i ||l_i||_{L^2(\R^3)}^2 =   ||l||_{L^2(\R^3)}^2 = \frac{1}{d_N^3} \sum_i |C(\rp_i)|^2 $ which yields
\begin{align*}
    \sum_{i}||A_3^i||_{L^2(B_i)}^2  \leq C\frac{a^9}{d_N^6} \sum_i |C(\rp_i)|^2 \leq C \lambda^{3}
\end{align*}
which entails $\mS \leq C \sum_{k=1}^3 ||A^i_k||_{L^2(B_i)}^2 \leq C \lambda^3$. Altogether, we proved that 
\begin{align*}
           ||D(v_N)||_{L^2(\R^3)}^2 \leq C \sum_{i} ||D(h_i)||_{L^2(B_i)}^2 \leq C( \mS + \mR )\leq C \lambda^3.
\end{align*}
\end{proof}

\section{Proof of Theorem \ref{main_theorem}}
\label{section::proof of theorem 1}

Let $\phi$ a divergence-free test function in $C^\infty_0(\R^3)$. We have using the splitting $u_N=u_N^{p} + u_N^{app} + v_N$
\begin{align}
    \int_{\R^3} D(u_N):D(\phi) = \int_{\R^3} D(u_N^{\frak{p}}):D(\phi)  + \int_{\R^3} D(u_N^{app}):D(\phi)  + \int_{\R^3} D(v_N):D(\phi).
    \label{FV u_N dem}
\end{align}
Owing to Section \ref{section::boundeness u_N}, we know there exists an accumulation point $u_\lambda$ of $(u_N)$ in $L^p_{loc}(\R^3)$ for $1<p<\frac{3}{2}$. Up to an extraction $N_k$, there exists $u_\lambda^{\frak{p}}$ an accumulation point of $(u_N)$ in $\dot{H}^1(\R^3)$ such that 
\begin{align*}
    u_\lambda = u_\lambda^{\frak{p}} + w_0 + \underset{k \rightarrow + \infty}{\lim} \, v_{N_k}
\end{align*}
where $w_0=  St^{-1}(\lambda \, \div \sigma_1)$ is the limit of $u_N^{app}$ in $L^p_{loc}(\R^3)$, see Proposition \ref{convergence u_app^N}. Passing to the limit along this extraction in \eqref{FV u_N dem} yields
\begin{multline*}
    2 \mu \int_{\R^3} D(u_\lambda):D(\phi) \\
    = 2\mu\int_{\R^3} D(u_\lambda^{\frak{p}}):D(\phi)  + 2 \mu  \int_{\R^3} D( w_0):D(\phi)  + \underset{k \rightarrow + \infty}{\lim} \, 2 \mu \int_{\R^3} D(v_{N_k}):D(\phi).
\end{multline*}
Note that this last limit is well-defined since all the other terms in  \eqref{FV u_N dem}  converges along the extraction $N_k$. According to the passive Theorem \ref{thm::DGV_Mecherbet}, $u_\lambda^{\frak{p}}$ is a solution to system \eqref{system::effective_passive} which entails
\begin{align*}
    2 \mu\int_{\R^3} D(u_\lambda^{\frak{p}}):D(\phi) = -2 \mu \int_{\R^3} \frac{5}{2} \rho \lambda  D(u_\lambda^{\frak{p}}):D(\phi) + \int_{\R^3} g \cdot \phi + \langle \mR_\lambda^p| \phi \rangle.
\end{align*}
By definition of $w_0$, we also have 
\begin{align*}
     2 \mu \int_{\R^3} D( w_0):D(\phi) = \lambda \int_{\R^3}  \sigma_1 : D(\phi).
\end{align*}
Altogether, we can write the following variationnal formulation for $u_\lambda$
\begin{align}
     2   \mu \int_{\R^3} (1+\frac{5}{2} \rho \lambda)  D(u_\lambda):D(\phi) = \int_{\R^3} g \cdot \phi + \lambda \int_{\R^3}  \sigma_1 : D(\phi) + \langle \mR_\lambda| \phi \rangle
     \label{FV::final u_lambda}
\end{align}
where 
\begin{align*}
   \langle \mR_\lambda | \phi \rangle & :=\langle \mR_\lambda^p | \phi \rangle + 2 \mu \int_{\R^3} (\frac{5}{2} \rho \lambda )  D(w_0): D(\phi) \\
   & + \underset{k \rightarrow + \infty}{\lim} \, 2 \mu \int_{\R^3} (\frac{5}{2} \rho \lambda )  D(v_{N_k}): D(\phi) + \underset{k \rightarrow + \infty}{\lim} \, 2 \mu \int_{\R^3} D(v_{N_k}): D(\phi).
\end{align*}
Let $q\geq 3$. A standard $W^{1,q'}(\R^3)$ estimate on the Stokes equations \eqref{systeme::w_0} satisfied by $w_0$ yields the estimate $|| D(w_0)||_{L^{q'}(\R^3)} \leq C \lambda$ for $q \geq 3$. Therefore, we get using Hölder inequality
\begin{align}
    \Big| 2 \mu \int_{\R^3} (\frac{5}{2} \rho \lambda )  D(w_0): D(\phi) \Big| \leq C \lambda^{2} ||D(\phi)||_{L^q(\R^3)}.
    \label{inequality w_0 in main theorem}
\end{align}
Thanks to the inequality \eqref{inequality::nouvelle useful limsup}, we know that
\begin{align*}
    \Big| \underset{k \rightarrow + \infty}{\lim} \, 2 \mu \int_{\R^3} D(v_{N_k}): D(\phi) \Big| \leq C \lambda^{\frac{5}{3}} ||D(\phi)||_{L^q(\R^3)}
\end{align*}
and we immediately get 
\begin{align*}
      \Big| \underset{k \rightarrow + \infty}{\lim} \, 2  \mu \int_{\R^3} (\frac{5}{2} \rho \lambda)  D(v_{N_k}): D(\phi) \Big| \leq C \lambda^{\frac{8}{3}} ||D(\phi)||_{L^q(\R^3)}.
\end{align*}
By density, these three inequalities hold for any divergence-free $\phi$ in $\dot{H}^1(\R^3) \cap \dot{W}^{1,q}(\R^3)$. Furthermore, we already know that $|\langle \mR_\lambda^p | \phi \rangle| \leq C \lambda^{\frac{5}{3}} ||D(\phi)||_{L^q(\R^3)}$ for any divergence-free $\phi$ in $\dot{H}^1(\R^3) \cap \dot{W}^{1,q}(\R^3)$, which gives overall the expected estimate
    \begin{align*}
    & |\langle  \mR_\lambda | \phi \rangle| \leq C \lambda^{\frac{5}{3}} ||D(\phi)||_{L^q(\R^3)}, \quad \forall \phi \in \dot{H}^1(\R^3) \cap \dot{W}^{1,q}(\R^3) \text{ divergence-free}.
\end{align*}
The proof is ended as \eqref{FV::final u_lambda} is exactly the weak formulation of the system \eqref{system::effective_main_theorem}.
\qed
\section{Comments on the results}
\label{section::comments}
We proved that the system of interest \eqref{system::main} may be approached, up to an error of order $\lambda^{\frac{5}{3}}$, by the following effective system 
\begin{equation}
    \left\{
      \begin{aligned}
         - 2 \mu \, \div ([1+\frac{5}{2}\rho \lambda] D(u_\lambda)) + \na p_\lambda & = g + \lambda \, \div{} \sigma_1  &&\text{ in }  \R^3,\\
         \div u_\lambda & = 0 &&\ \text{in} \ \R^3,\\
         \underset{|x| \rightarrow + \infty}{ \lim} u_\lambda(x) &  = 0.
      \end{aligned}
    \right.
    \label{system::effective_model_comment_section}
\end{equation}
with $\sigma_1(x)= \alpha \mJ \int_{\mS^2} (\rp \otimes \rp-\frac{1}{3} Id) f(x,\rp) \, \d \rp. $ The modified effective viscosity of Einstein $ \mu(1+\frac{5}{2} \rho \lambda)$ naturally appears on the left-hand side and corresponds to the particles passive contribution. However, particles self-propulsion is responsible for the new stress source term $\lambda \, \div \sigma_1$ on the right-hand side, where $\sigma_1$ is the classic \textit{active stress tensor}. As explained in the introduction, several works explored active suspensions rheology, and it was shown that some anisotropy is needed to observe the effect of activity. {} Note that other effects, different from particles activity, can be responsible for similar stress terms. For instance, a similar viscoelastic stress is derived in \cite{hofer2022derivation} for passive non-spherical Brownian particles. \black{}  \medskip

In order to compare our result to these studies, remember that Stokes models of the form \eqref{system::effective_model_comment_section} are obtained from the incompressible Navier-Stokes equations at low Reynold number, in a stationary state at a frozen time $t$. With that in mind, we define an instantaneous kinetic energy of the flow $u_\lambda$ in a domain $\Omega $ containing $\mO$ 

\begin{equation*}
    E_c(t):=\frac{1}{2} \int_{\Omega} |u_\lambda|^2.
\end{equation*}
Performing an energy balance as in \cite{saintillan2010extensional}, the instantaneous energy dissipation rate is given by 
\begin{equation*}
    \frac{\mathrm{d} E_c}{\mathrm{d}t}(t) = -2 \mu \int_{\Omega} (1+\frac{5}{2} \rho \lambda) |D(u_\lambda)|^2  - \lambda \int_{\Omega} \sigma_1 : D(u_\lambda).
\end{equation*}
The first term corresponds to the mechanical energy loss due to the particles presence, and it appears that the Einstein contribution to viscosity leads to a greater energy dissipation. The second term is the one of interest here as it reflects the impact of particles self-propulsion in the system energy. \medskip

In more complete kinetic models where the particles dynamics is taken into account, the probability density $f(x,\rp,t)$ usually depends on time and solves a Fokker-Planck equation coupled with the Stokes one. Namely, the stationary fluid system \eqref{system::effective_model_comment_section} forms a closed system when coupled with the following Fokker-Planck equation
\begin{equation}
    \label{equation::FP}
    \partial_t f(x,\rp,t) + \na_x \cdot (\dot{x} \, f(x,\rp,t)) + \na_\rp \cdot (\dot{\rp} \, f(x,\rp,t)) = 0
\end{equation}
where $\dot{x}$ and $\dot{\rp}$ are the translational and angular flux velocities. A popular example is due to Saintillan and Shelley \cite{saintillan2008instabilities,alonso2016microfluidic}, and is given by
\begin{align}
    \dot{x} & = U_0 \rp + u_\lambda(x,t)\\
    \dot{\rp} & =  (I_d - \rp \otimes \rp) [ \xi D(u_\lambda(x,t)) + W(u_\lambda(x,t) )]\rp - D_r \na_{\rp} \ln f(x,\rp,t).
\end{align}
In the above equations, $U_0$ is the translational velocity of an isolated particle and $\xi$ is a parameter depending on the particles shape. $D_r$ denotes the rotational Brownian diffusion coefficient, and $W(u_\lambda(x,t))$ is the vorticity tensor of the fluid velocity $u_\lambda(x,t)$. In the greatly simplified model where we neglect the transport effects as well as vorticity, equation \eqref{equation::FP} takes the form
\begin{align}
\label{equation::FP simplified}
    \partial_t f + \na_\rp \cdot ( (I_d - \rp \otimes \rp) \xi D(u_\lambda(x,t)) \rp f)  - D_r \Delta_\rp f= 0.
\end{align}
Considering $f(x,\rp)$ a stationnary solution of the above equation, we can look for a probability density for which spatial dependence only occurs through the flow symmetric gradient, that is $f(x,\rp):= F[D(u_\lambda)(x)](\rp)$. In other words, $x$ can be seen as a parameter only acting through $D(u_\lambda)(x)$ in equation \eqref{equation::FP simplified}. In that specific case, the energy gain or loss due to self-propulsion is given by
\begin{align}
    - \lambda \int_{\Omega} \sigma_1 : D(u_\lambda) = -C \alpha \int_{\Omega} \Big( \int_{\mS^2} \rp \otimes \rp : D(u_\lambda)(x) F[D(u_\lambda)(x)](\rp) \, \d \rp \Big) \, \d x.
    \label{formula::loss/gain_energy}
\end{align}
Several remarks are in order about this formula. \medskip

Assuming some symmetry on the probability density $F$, say for any trace-free symmetric matrix $S$, we have $F[S](\rp_1,\rp_2,\rp_3) = F[S](\rp_1,-\rp_2,\rp_3)$ which means that particles spend an equal amount of time oriented into two opposed hemisphere. In that case,  the active contribution in \eqref{formula::loss/gain_energy} vanishes to zero when averaging over $\mS^2$. This is in accordance with many studies \cite{ishikawa2007rheology,haines2008effective, decoene2019modelisation} who pointed out the absence of rheology signature for isotropic distribution of active spherical particles. \medskip

At the opposite, if we want the active contribution  to have a fixed sign, we need to find a condition on $F[S](\rp)$ such that for any trace-free symmetric matrix $S$:
\begin{equation}
    \int_{\mS^2}  \rp \otimes \rp : S \, F[S](\rp) \, \d \rp >0.
    \label{condition::anisotropie_densite}
\end{equation}
If this condition is verified and if we consider a  pusher-like particles suspension ($\alpha<0$), the term $- \lambda \int_{\Omega} \sigma_1 : D(u_\lambda)$ injects kinetic energy in the system and is opposed to the energy loss due to the effective viscosity $(1+ \lambda \frac{5}{2} \rho )\mu$, hence a diminution in the measurement of the apparent viscosity which is consistent with the experience carried out in \cite{sokolov2009reduction}. For puller-like particles, the effect is opposite and more mechanical energy is being dissipated by the activity, thus an apparent viscosity augmentation, which is confirmed by the experiments described in \cite{rafai2010effective}. It remains to understand how condition \eqref{condition::anisotropie_densite} can be linked to some anisotropy on the orientation probability density. We give now a simple condition based on the eigenvalues of $S$. Consider $\lambda^+$ a positive eigenvalue of $S$, say the greatest, and $\rp^+$ the associated eigenvector in $\mS^2$.  As in \cite{haines2008effective, saintillan2008instabilities}, we consider a suspension of nearly aligned particles in the $\rp^+$ direction
\begin{equation}
    F[S](\rp)= \delta(\rp-\rp^+)
    \label{assumption::probability_density_dirac}
\end{equation}
and we verify \eqref{condition::anisotropie_densite} as
\begin{align*}
     \int_{\mS^2} \rp \otimes \rp : S  \, F[S](\rp) \, \d \rp = |\mS^2|\rp^+ \otimes \rp^+ : S = |\mS^2| \lambda^+ > 0.
\end{align*}
Looking back to the empirical measure convergence \eqref{assumption::empirical_measure}, we have for any $\varphi$ in $C^0(\R^3 \times \mS^2)$
\begin{equation*}
    \frac{1}{N} \sum_{i=1}^N \varphi(x_i,\rp_i) \underset{N \rightarrow + \infty}{\longrightarrow} \int_{\R^3}  \varphi(x,\rp_x^+) \, \d x
\end{equation*}
where $\rp_x^+$ is the normalized eigenvector associated to the biggest positive eigenvalues of $D(u_\lambda)(x)$. This means that particles at location $x$ tend to orient themselves in the extensional direction $\rp_x^+$ of the local flow $D(u_\lambda)(x)$, which is a classical feature of active suspensions.

\section{Relaxed assumptions}
 \label{section::Relaxed Assumption}
 As in \cite{GVRH}, the separation assumption \eqref{assumption::separation} is relaxed and replaced by the followings
 \begin{align}
     & \text{ there exists } M>2 \beta, \text{ such that }\forall N, \  \forall i \neq j, \ |x_i - x_j| \geq M a \tag{H2'} \label{assumption::separation_relaxed_1},\\
     & \text{ there exists } C,\alpha >0, \text{ such that } \forall \eta >0  , \ \sharp \{i, \ \exists j \neq i, \ |x_i- x_j| \leq \eta N^{-1/3} \} \leq C \eta^\alpha N. \tag{H2''} \label{assumption::separation_relaxed_2}
 \end{align}
The first assumption ensures that the micro-swimmers do not overlap. The second one means that a controlled number of particles might be very close to each other. For a fixed $\eta>0$, we introduced the following set of good and bad indices 
\begin{align*}
    & G_\eta := \{ 1 \leq i \leq N, \ \forall j \neq i, \  |x_i - x_j| \geq \eta N^{-1/3} \}, \\
    & B_\eta :=\{ 1, \dots, N\} \setminus G_\eta.
\end{align*}
The purpose of this Section is to see how slight modifications of the above analysis enables to replace Theorem \ref{main_theorem} by the following 
\begin{theorem}
\label{main_theorem relaxed duality}
Assume \eqref{assumption::empirical_measure_spatial}-\eqref{assumption::separation_relaxed_1}-\eqref{assumption::separation_relaxed_2} and that $g_N \rightarrow g$ in $L^{\frac{6}{5}}(\R^3)$ with $g \in L^1(\R^3) \cap L^\infty(\R^3)$.  Then there exists $q_{min}>1$ such that any accumulation point $u_\lambda$ of $(u_N)_N$ in  $L^p_{loc}(\R^3)$ ($1 < p < \frac{3}{2}$) verifies in the sense of distributions
\begin{equation*}
    \left\{
      \begin{aligned}
         - 2 \mu \, \div ([1+\frac{5}{2}\rho \lambda] D(u_\lambda)) + \na p_\lambda & = g + \lambda \, \div \sigma_1 + \mR_\lambda &&\ \text{in} \ \R^3,\\
         \div u_\lambda &  = 0 &&\ \text{in} \ \R^3,\\
         \underset{|x| \rightarrow + \infty}{ \lim} u_\lambda(x)  & = 0
      \end{aligned}
    \right.
\end{equation*}
where for all $q \geq q_{min}$ there exists $\delta<0$ such that the remainder verifies
\begin{align*}
    & \langle  \mR_\lambda | \phi \rangle \leq C \lambda^{1+\delta} ||D(\phi)||_{L^q(\R^3)}, \quad \forall \phi \in \dot{H}^1(\R^3) \cap \dot{W}^{1,q}(\R^3) \text{ divergence-free}
\end{align*}
and $\sigma_1$ defined in \eqref{def::sigma_1}.
\end{theorem}
Using the set of indices $G_\eta$ and $B_\eta$, Gérard-Varet and Höfer \cite{GVRH} approached the passive system \eqref{system::main_passive} by an Einstein-like effective system of order one. More precisely, we have the following result that replaces Theorem \ref{thm::DGV_Mecherbet} in the analysis of the passive system \ref{system::main_passive}
\begin{theorem}[Theorem 1 from \cite{GVRH}]
\label{theoreme DGV Höfer ordre 1}
Let $\lambda>0$. Assume \eqref{assumption::empirical_measure_spatial}-\eqref{assumption::separation_relaxed_1}-\eqref{assumption::separation_relaxed_2} and that $g_N \rightarrow g$ in $L^{\frac{6}{5}}(\R^3)$ with $g \in L^1(\R^3) \cap L^\infty(\R^3)$. Then there exists $q_{min}^{\frak{p}}>1$ and $C>0$ such that any accumulation point $u_\lambda^{\frak{p}}$ of $(u_N^{\frak{p}})_N$ in $\dot{H}^1(\R^3)$ is a solution in the sense of distributions of
\begin{equation*}
    \left\{
      \begin{aligned}
         - 2 \mu \, \div ([1+\frac{5}{2}\rho \lambda] D(u_\lambda^{\frak{p}})) + \na p_\lambda^{\frak{p}} & = g + \mR^{\frak{p}}_\lambda &&\ \text{in} \ \R^3,\\
         \div u_\lambda^{\frak{p}} &  = 0 &&\ \text{in} \ \R^3,\\
         \underset{|x| \rightarrow + \infty}{ \lim} u_\lambda^{\frak{p}}(x)  & = 0
      \end{aligned}
    \right.
    \label{effective_finale system_theorem passive Höfer}
\end{equation*}
where for all $q \geq q_{min}^{\frak{p}}$, there exists $\delta^{\frak{p}}>0$ such that
\begin{align*}
    & \langle  \mR^{\frak{p}}_\lambda | \phi \rangle \leq C \lambda^{1+\delta^{\frak{p}}} ||D(\phi)||_{L^q(\R^3)}, \quad \forall \phi \in \dot{H}^1(\R^3) \cap \dot{W}^{1,q}(\R^3) \text{ divergence-free.}
\end{align*}
\end{theorem}
\begin{remark}
\label{remarque exposents}
The constants are explicit and given by $q_{min}^{\frak{p}}:= \frac{6 + 2 \alpha}{\alpha}$ and $\delta^{\frak{p}}:=\frac{q-1}{q}-\frac{6(q-1)}{6(q-1)+(q-2) \alpha}$. For the estimate on $ \mR^{\frak{p}}_\lambda$ to hold, the parameter $\eta$ was assumed to be of the form $\lambda^{\theta}$ with $\theta:= \frac{2 q}{6 (q-1) + (q-2) \alpha}$ which always lies in $]0,\frac{1}{3}[$. The same choice will be made for Theorem \ref{main_theorem relaxed duality}.
 \end{remark}
 The strategy of the global analysis remains the same. We approximate $u_N^{\frak{a}}=u_N^{app}+v_N$ with the same dilute approximation $u_N^{app}$ defined in \eqref{def::u_N^app_regime_dilué}. All the results presented in Section \ref{section::boundeness u_N} remains true in this new setting. In particular, we know that $u_N^{app}$ converges toward $w_0$ in $L^p_{loc}(\R^3)$ and that $(u_N)_N$ is bounded in $L^p_{loc}(\R^3)$. However, the new separation assumptions lead to differences in the analysis of the error term $v_N$ which we detail below. Lastly, the proof of Theorem \ref{main_theorem relaxed duality} is carried out using a similar approach. \\
 \medskip

 \noindent The solution $v_N$ to \eqref{systeme::error_v_N} corresponds to the error made by the Method of Reflections and is here to correct the trace of $u_N^{app}$ on the particles through the boundary condition $D(v_N)= -D(h_i)$ on particle $B_i$ with  $ h_i(x) = \sum_{j \neq i} v[\rp_j](x-x_j)$. We will split the influence of good and bad particles and we write $v_N:= v_N^{1}+ v_N^{2}$ where \begin{align*}
    St(v_N^1) = 0 \quad \text{in} \quad  \Omega_N, \quad D(v_N^1)=  -D \big( \sum_{j \neq i \atop j \in G_\eta} v[\rp_j](x-x_j) \big):= - D(  h_i^1(x))\quad \text{on} \quad B_i
\end{align*} 
and 
\begin{align*}
    St(v_N^2) = 0 \quad \text{in} \quad  \Omega_N, \quad D(v_N^2)=  -D \big( \sum_{j \neq i \atop j \in B_\eta} v[\rp_j](x-x_j) \big):= - D(  h_i^2(x))\quad \text{on} \quad B_i.
\end{align*}
\subsubsection{Analysis of $v_N^1$}
 The proof of the following Proposition is similar to the proof of Proposition \ref{prop:limsup v_N} and is still valid as long as particles do not overlap, which holds under assumptions \eqref{assumption::separation_relaxed_1}.
 \begin{proposition}
For all $q \geq 2$ and all divergence-free $\phi$ in $C^\infty_0(\R^3)$,
\begin{align*}
    \underset{N \rightarrow + \infty}{\limsup} \, \Big| 2 \mu \int_{\R^3} D(v_N^1):D(\phi) \Big| \leq C \lambda^{\frac{1}{2}-\frac{1}{q}} ||D( v_N^1)||_{L^2(\R^3)} ||D(\phi)||_{L^q(\R^3)}.
\end{align*}
\label{prop:limsup v_N section 6}
\end{proposition}
 The next result replaces Proposition \ref{prop_borne_uniforme_v_N} in the analysis. The proof must be modified to take into account the new separation assumptions. 
 \begin{proposition}
For all $N >0$ and $q \geq 2$,
\begin{align*}
            ||D(v_N^1)||_{L^2(\R^3)} \leq C \big(\lambda^{\frac{3}{2}} \eta^{-3}  \:+\:   \lambda^{\frac{1}{6}} \eta^{1+\frac{\alpha}{2}-\frac{\alpha}{q}} \:+\: \lambda^{\frac{3}{2} - \frac{1}{q}} \eta^{-3+\frac{\alpha}{2}+\frac{3}{q}-\frac{\alpha}{q}}  \big).
\end{align*}
\label{prop_borne_uniforme_v_N section 6}
\end{proposition}
We will need the following Lemma, proved in the dedicated annex section.
\begin{lemma}
Using the notation $y_i = x_i N^{\frac{1}{3}}$, we denote $\tilde{M} = \frac{3 M}{4 \pi}$  such that assumptions \eqref{assumption::separation_relaxed_1}-\eqref{assumption::separation_relaxed_2} read
\begin{align*}
     \min_{i \neq j} |y_i - y_j| & \geq \tilde{M} \lambda^{\frac{1}{3}}\\
    \min_{i \in G_\eta, \, j \neq i} |y_i - y_j| & \geq \eta.
\end{align*}
There exists a constant $C>0$ such that the following inequalities hold
\begin{align}
     \sup_{i \in G_\eta} \sum_{j \in G_\eta} (\eta + |y_i-y_j|)^{-4} & \leq C \eta^{-4} \label{lemme ine 1}\\
     \sup_{i} \sum_{j} (\tilde{M} \lambda^{\frac{1}{3}} + |y_i-y_j|)^{-4} & \leq C \lambda^{-\frac{4}{3}}. \label{lemme ine 4}
\end{align}
\label{lemma comparaison série intégrale}
\end{lemma}

\begin{proof}[Proof of Proposition \ref{prop_borne_uniforme_v_N section 6}]
The first steps of the proof are the same as in Proposition \ref{prop_borne_uniforme_v_N} and thanks to a variational argument we get $ ||D(v_N^1)||_{L^2(\R^3)}^2 \leq  \sum_{i} ||D(h_i^1)||_{L^2(B_i)}^2 $. Through Hölder inequality and the Taylor expansion \eqref{taylor expansion D(h_i)}, we write
\begin{align*}
       & \sum_{i} ||D(h_i^1)||_{L^2(B_i)}^2 = \sum_{i \in G_\eta} ||D(h_i^1)||_{L^2(B_i)}^2 + \sum_{i \in B_\eta} ||D(h_i^1)||_{L^2(B_i)}^2\\
    & \leq  C    \Big( \sum_{i \in G_\eta} |B_i|\Big)^{1-\frac{2}{q}} \: \Big( \sum_{i \in G_\eta}  \int_{B_i} \Big| a^3 \sum_{j \neq i \atop j \in G_\eta} \M(x-x_j) C(\rp_j) \Big|^q  \d x \Big)^{\frac{2}{q}} \black \\
    & + C  \Big( \sum_{i \in B_\eta} |B_i|\Big)^{1-\frac{2}{q}} \: \Big( \sum_{i \in B_\eta}  \int_{B_i} \Big| a^3 \sum_{j \neq i \atop j \in G_\eta} \M(x-x_j) C(\rp_j) \Big|^q  \d x \Big)^{\frac{2}{q}}  \\
    & +C  \Big( \sum_{i \in G_\eta} |B_i|\Big)^{1-\frac{2}{q}} \: \Big( \sum_{i \in G_\eta} \int_{B_i} \Big| \sum_{j \neq i \atop j \in G_\eta} D(R[\rp_j])(x-x_j)\Big|^q  \d x \Big)^{\frac{2}{q}}\\
    & + +C  \Big( \sum_{i \in B_\eta} |B_i|\Big)^{1-\frac{2}{q}} \: \Big( \sum_{i \in B_\eta} \int_{B_i} \Big| \sum_{j \neq i \atop j \in G_\eta} D(R[\rp_j])(x-x_j)\Big|^q  \d x \Big)^{\frac{2}{q}}\\
    & \leq C \lambda^{1-\frac{2}{q}}(  \mS_1^{\frac{2}{q}} \black + \eta^{\alpha - \frac{2 \alpha}{q}} \mS_2^{\frac{2}{q}}) + C \lambda^{1-\frac{2}{q}}(\mR_1^{\frac{2}{q}} + \eta^{\alpha - \frac{2 \alpha}{q}} \mR_2^{\frac{2}{q}}) 
\end{align*}
To bound the quantity $\mS_1:= \sum_{i \in G_\eta}  \int_{B_i} \Big| a^3 \sum_{j \neq i \atop j \in G_\eta} \M(x-x_j) C(\rp_j) \Big|^q  \d x $, we introduce
\begin{align*}
    d_\eta:= \frac{\eta}{4} N^{-\frac{1}{3}} \leq \frac{1}{4} \, \underset{i \neq j \atop i \in G_\eta}{\min} \, |x_i - x_j|
\end{align*}
and the function $l_G$ defined by
\begin{align*}
    l_G:=  \sum_{j \in G_\eta} C(\rp_j) |B(x_j,d_\eta)|^{-1} 1_{B(x_j,d_\eta)}.
\end{align*}
For any $i \in G_\eta$ and any $x \in B_i$, we have the decomposition
\begin{align*}
    a^3 \sum_{j \neq i \atop j \in G_\eta} \M(x-x_j) C(\rp_j) & = a^3 \int_{\R^3 \setminus B(x_i,2 d_\eta)} (\M(x-T(y)) - \M(x-y)) l_G(y) \, \d y\\
    & + a^3 \dashint_{B(x_i,d_\eta) } \int_{\R^3 \setminus B(x_i,2 d_\eta)}  (\M(x-y) - \M(z-y)) l_G(y) \, \d y \d z\\
    & + a^3 \dashint_{ B(x_i,d_\eta) } \int_{\R^3 \setminus B(x_i,2 d_\eta)}  \M(z-y) l_G(y) \, \d y \d z \black \\
    & := A^1_{i}(x) + A^2_{i}(x) +  A^3_{i} .
\end{align*}
We start with the analysis of $ A^1_{i}(x)$ and we wish to compute $ \sum_{i \in G_\eta} ||A^1_{i}||_{L^q(B_i)}^q$. For $y \in B(x_j,d_\eta)$ with $j \in G_\eta, \ j \neq i$, we write
\begin{align*}
    |\M(x-x_j) - \M(x-y)| \leq C |y-x_j| \sup_{\bar{z} \in [x_j,y]} |\bar{z}-x|^{-4} \leq C d_\eta (|x_i-x_j|- d_\eta - a)^{-4},
\end{align*}
and using assumptions \eqref{assumption::separation_relaxed_1}-\eqref{assumption::separation_relaxed_2} yields
\begin{align*}
     (|x_i-x_j|- d_\eta - a)^{-4} \leq |x_i-x_j|^{-4} \Big(\frac{3}{4}-\frac{1}{M}\Big)^{-4} \leq C|x_i-x_j|^{-4}
\end{align*}
as long as $M>\frac{4}{3}$. Therefore, we obtain, using that $|y_i-y_j| \geq \frac{1}{2} (|y_i-y_j| + \eta )$ for $i \neq j $ in $G_\eta$
\begin{align*}
    \sum_{i \in G_\eta} ||A^1_{i}||_{L^q(B_i)}^q & \leq  C  a^{3q} \sum_{i \in G_\eta} \int_{B_i} \Big| \sum_{j \neq i, \, j \in G_\eta}  \frac{d_\eta |C(\rp_j)|}{|x_i-x_j|^4} \Big|^q \, \mathrm{d}x\\
    & \leq  C  a^{3 + 3q} \eta^q N^{-\frac{q}{3}} N^{\frac{4q}{3}} \sum_{i \in G_\eta} \Big| \sum_{j \neq i, \, j \in G_\eta}  (\eta + |y_i-y_j|)^{-4} |C(\rp_j)| \Big|^q\\
    & \leq C  \frac{\lambda^{q+1}}{N} \eta^{q} \Big( \sup_{i \in G_\eta} \sum_{j \in G_\eta} (\eta+|y_i-y_j|)^{-4} \Big)^q  \sum_{i \in G_\eta} |C(\rp_i)|^q\\
    & \leq C \lambda^{q+1} \eta^{-3q}
\end{align*}
where we used inequality \eqref{lemme ine 1}.
The analysis of $A^2_{i}$ is very similar and yields the same estimate.  To deal with $A^3_{i}$, we introduce the following notation:
\begin{align*}
    l_G^{i}= 
\begin{cases}
    C(\rp_i) |B(x_i,d_\eta)|^{-1} 1_{B(x_i,d_\eta)}  \quad & \text{if} \quad i \in G_\eta ,\\
    0             &\text{otherwise}
\end{cases}.
\end{align*}
For all $z \in B(x_i, d_\eta)$, we can write
\begin{align*}
    \int_{\R^3 \setminus B(x_i,2 d_\eta)}  \M(z-y) l_G(y) \, \d y = (\M \star (l_G-l_G^{i}))(z)
\end{align*}
and we subsequently obtain
\begin{align*}
    \sum_{i \in G_\eta}||A^3_{i}||_{L^q(B_i)}^q & = C \frac{a^{3 + 3q}}{d_\eta^{3q}} \sum_{i \in G_\eta} \Big(\int_{B(x_i,d_\eta)} (\M \star (l_G-l_G^i))(z) \, \d z \Big)^q\\
    & \leq C \frac{a^{3 + 3q}}{d_\eta^{3q}} \sum_{i \in G_\eta}|| \M  \star (l_G-l_G^i)||_{L^q(B(x_i,d_\eta))}^{q} |B(x_i,d_\eta)|^{q-1}\\
    & \leq C  \frac{a^{3 + 3q}}{d_\eta^{3}} \sum_{i \in G_\eta} || \M  \star (l_G-l_G^i)||_{L^q(B(x_i,d_\eta))}^{q}\\
    & \leq  C  \frac{\lambda^{q+1}}{N^{q+1}} \eta^{-3} \big( ||\M  \star l_G||_{L^q(\R^3)}^{q} + \sum_{i \in G_\eta} || \M \star l_G^{i}||_{L^q( \R^3)}^{q} \big)
\end{align*}
where we used that the family of balls $(B(x_i,d_\eta))_{i \in G_\eta}$ is pair-wise disjoint thanks to \eqref{assumption::separation_relaxed_2}. Since $\M$ is a Calderon-Zyngmund operator, we have
\begin{align*}
    \sum_{i \in G_\eta}||A^3_{i}||_{L^q(B_i)}^q  \leq \frac{\lambda^{q+1}}{N^{q}} \eta^{-3} \big( \, ||l_G||_{L^q( \R^3)}^{q} + \sum_{i \in G_\eta}||l_G^i||_{L^q( \R^3)}^{q} \big).
\end{align*}
    It remains to compute 
    \begin{align*}
  \sum_{i \in G_\eta}||l_G^i||_{L^q( \R^3)}^{q}  = ||l_G||_{L^q(\R^3)}^{q} & = \int_{\R^3} \Big| \sum_{j \in G_\eta} C(\rp_j) |B(x_j,d_\eta)|^{-1} 1_{B(x_j,d_\eta)} \Big|^q \, \mathrm{d}x\\
   & = |B(x_j,d_\eta)|^{-q} \sum_{j \in G_\eta} \int_{B(x_j,d_\eta)} |C(\rp_j)|^q \, \d x\\
   & \leq C \eta^{3 - 3 q} N^{q} 
\end{align*}
 which implies $\sum_{i}||A^3_{i}||_{L^q(B_i)}^q \leq C \lambda^{q+1} \eta^{-3q} $. Altogether, we proved that $\mS_1 \leq C \lambda^{q+1} \eta^{-3q}$.\\

\noindent To estimate $\mS_2:=\sum_{i \in B_\eta}  \int_{B_i} \Big| a^3 \sum_{j \neq i \atop j \in G_\eta} \M(x-x_j) C(\rp_j) \Big|^q  \d x $, we introduce the following notation
\begin{align*}
    d_M:= \frac{M a}{4} \leq \frac{1}{4} \, \underset{i \neq j}{\min} \, |x_i-x_j|.
\end{align*}
We use a similar decomposition for any $i \in B_\eta$ and $x \in B_i$: 
\begin{align*}
    a^3 \sum_{j \in G_\eta} \M(x-x_j) C(\rp_j) &  = a^3 \int_{\R^3} (\M(x-T(y)) - \M(x-y)) l_G(y) \, \d y\\
    & + a^3 \dashint_{B(x_i,d_M) } \int_{\R^3}  (\M(x-y) - \M(z-y)) l_G(y) \, \d y \d z\\
    & + a^3 \dashint_{ B(x_i,d_M) } \int_{\R^3 }  \M(z-y) l_G(y) \, \d y \d z\\
    & := D^1_{i}(x) + D^2_{i}(x) + D^3_{i}.
\end{align*}
We first deal with $D^1_{i}$. For all $y \in B(x_j,d_\eta)$ with $j \in G_\eta$, we have $|\M(x-x_j) - \M(x-y)| \leq C d_\eta |x_i - x_j|^{4}$, thus
\begin{align*}
    \sum_{i \in B_\eta} ||D^1_{i}||_{L^q(B_i)}^q &  \leq a^{3 + 3q} d_\eta^{q} N^{\frac{4q}{3}} \sum_{i \in B_\eta} \Big| \sum_{j \in G_\eta} (\tilde{M} \lambda^{\frac{1}{3}} + |y_i-y_j|)^{-4} |C(\rp_j)| \Big|^q\\
    & \leq \frac{\lambda^{q+1}}{N} \eta^{q} \Big( \sup_{i } \sum_{j} (\tilde{M} \lambda^{\frac{1}{3}}+ |y_i-y_j|)^{-4} \Big)^q \sum_{i \in G_\eta} |C(\rp_i)|^q\\
    & \leq C \lambda^{1-\frac{q}{3}} \eta^{q} \quad \text{using} \quad \eqref{lemme ine 4}.
\end{align*}
The same estimate holds for $\sum_{i \in B_\eta} ||D^2_{i}||_{L^q(B_i)}^q$.  We now turn to the contribution from $D^3_{i}$ and we write
\begin{align*}
     \sum_{i \in B_\eta} ||D^3_{i}||_{L^q(B_i)}^q &  \leq C \frac{a^{3 + 3q}}{d_M^{3q}} \sum_{i \in B_\eta} \Big( \int_{B(x_i,d_M)} (\M \star l_G )(z) \, \d z \Big)^q\\
     & \leq C \frac{a^{3 + 3q}}{d_M^3} \sum_{i \in B_\eta} || \M \star l_G||_{L^q(B(x_i,d_M))}^q\\
     & \leq C a^{3q} ||l_G||_{L^q(\R^3)}^q\\
     & \leq C \lambda^{q} \eta^{3-3q}.
\end{align*}
Gathering all the results together, there holds $\mS_2 \leq C(\lambda^{1-\frac{q}{3}} \eta^{q}  + \lambda^q \eta^{-3q+3})$. Comparable computations yield similar estimates for the remainder terms $\mR_1$ and $\mR_2$ and we obtain
\begin{align*}
          &  ||D(v_N^1)||_{L^2(\R^3)}^2  \leq C \lambda^{1-\frac{2}{q}} \: (\mS_1^{\frac{2}{q}} \: + \: \eta^{ \alpha - \frac{2 \alpha}{q}} \mS_2^{\frac{2}{q}})  \:+ \: C \lambda^{1-\frac{2}{q}}(\mR_1^{\frac{2}{q}} \: + \: \eta^{\alpha-\frac{2 \alpha}{q}} \mR_2^{\frac{2}{q}}) \\
           & \leq C \big(\lambda^3 \eta^{-6} \: + \:+\: \lambda^{\frac{1}{3}} \eta^{ 2+ \alpha-\frac{2 \alpha}{q} +2} \:+\: \lambda^{3 - \frac{2}{q}} \eta^{-6 + \alpha +\frac{6}{q} - \frac{2 \alpha}{q}} \big).
\end{align*}
\end{proof}
\subsubsection{Analysis of $v_N^2$}

Before studying $v_N^2$, we recall a few properties on an elementary straining-flow problem which is reminiscent of passive particles analysis. For a trace-free symmetric matrix $S$, we denote $w[S]$ the solution of
\begin{align*}
    - \mu \Delta w[S] + \na p[S] & = 0  && \text{in} \quad \R^3 \setminus B(0,a)\\
     \div w[S] & = 0  && \text{in} \quad \R^3 \setminus B(0,a)\\
     w[S] & = -S x  && \text{on} \quad \partial B(0,a)
\end{align*}
which is explicit \cite{guazzelli2011physical} and given by 
\begin{align*}
    w[S](x)& = \: - \frac{5}{2} S : (x \otimes x) \frac{a^3 x}{|x|^5} \: - \: S x \frac{a^5}{|x|^5} \:+ \: \frac{5}{2} (S : x \otimes x) \frac{a^5 x}{|x|^7}\\
    & = \: 5  \frac{\lambda}{N} \mu \nabla U(x) S \: + \: R[S](x)
\end{align*}
with $|R[S](x)| = \mathcal{O}(a^5 |x|^{-4})$. Note that this expansion has a structure similar to the expansion \eqref{taylor_expansion_v[p]} of $v[\rp]$. Precisely, we write using the linearity of $S \mapsto w[S]$:
\begin{align*}
    v[\rp](x) - \frac{1}{5} w[C(\rp)](x)= \tilde{R}[\rp](x)
\end{align*}
where $\tilde{R}[\rp]=R[\rp] - \frac{1}{5}R[C(\rp)]$ is a remainder satisfying 
$$|\tilde{R}[\rp](x)|= \mathcal{O} \big(  a^4 (|x|^{-3} \: + \:|x-a \beta \rp|^{-3}\:+\: |x-a \beta^{-1} \rp|^{-3})\:+ \: a^5 |x|^{-4} \big).$$
Inspired from this decomposition, we split the boundary condition of the field $v_N^2$ on the ball $B_i$ through   $$h_i^2(x) := \sum_{j \neq i} v[\rp_j](x-x_j) = \frac{1}{5}\sum_{j \neq i} w[C_{\rp_j}](x-x_j) + \sum_{j \neq i} \tilde{R}[\rp_j](x-x_j) \quad \text{for} \quad x \in B_i.$$
Subsequently, we write $v_N^2 :=  \: \tilde{v}_N^2 \: + \:  \psi_N$ where $\tilde{v}_N^2$ is a solution to 
\begin{align*}
    St(\tilde{v}_N^2) = 0 \quad \text{in} \quad \Omega_N, \quad D(\tilde{v}_N^2)(x) = - \frac{1}{5}D \big( \sum_{j \neq i \atop j \in B_\eta} w[C(\rp_j)](x-x_j) \big) \quad \text{in} \quad B_i
\end{align*}
and $\psi_N$ solves 
\begin{align*}
    St(\psi_N) = 0 \quad \text{in} \quad \Omega_N, \quad D(\psi_N)(x) = - D \big( \sum_{j \neq i \atop j \in B_\eta }\tilde{R}[\rp_j](x-x_j) \big) \quad \text{in} \quad B_i.
\end{align*}
The control of $\psi_N$ is obtained through
\begin{proposition}
For all $q \geq 2$ and all divergence-free $\phi$ in $C^\infty_0(\R^3)$, 
\begin{align*}
    \underset{N \rightarrow + \infty}{\limsup} \, \Big| 2 \mu \int_{\R^3} D(\psi_N):D(\phi) \Big| \leq C \lambda^{1-\frac{1}{q}} \eta^{\frac{\alpha}{2}}||D(\phi)||_{L^q(\R^3)}.
\end{align*}
\end{proposition}
\begin{proof}
We only explain the proof scheme. The first step is to apply a Bogovskii argument as we did in Proposition \ref{prop:limsup v_N}, it is then necessary to estimate the norm of $D(\psi_N)$ in $L^2(\R^3)$. As the problem solved by $\psi_N$ admits a simple minimization representation, we only need to control the data inside the balls, that is $\sum_{i} ||D \big( \sum_{j \neq i \atop j \in B_\eta }\tilde{R}[\rp_j](\cdot-x_j) \big) ||_{L^2(B_i)}^2$. Terms of this nature were already encountered, and a similar analysis yields 
\begin{align*}
    \sum_{i} ||D \big( \sum_{j \neq i \atop j \in B_\eta }\tilde{R}[\rp_j](\cdot-x_j) \big) ||_{L^2(B_i)}^2 \leq C \lambda \eta^\alpha.
\end{align*}
\end{proof}
In order to study $\tilde{v}_N^2$, we introduce the following sum of elementary solutions which relates again to the Method of Reflections
\begin{align*}
    \phi_N^{app}:= \frac{1}{5}\sum_{i \in B_\eta} w[C(\rp_i)](x-x_i)
\end{align*}
and we verify the following separation $\tilde{v}_N^2 = \phi_N - \phi_N^{app}$ where $\phi_N$ is a solution to 
\begin{align*}
     St(\phi_N) = 0 \quad \text{in} \quad \Omega_N, \quad D(\phi_N) = \begin{cases}
         & \frac{1}{5} C(\rp_i) \quad \text{in} \quad B_i, \quad \text{if} \quad i \in B_\eta\\
         & 0 \quad \text{in} \quad B_i, \quad \text{otherwise}
     \end{cases}.
\end{align*}
The analysis of $\phi_N$ is standard. Using a Bogovskii argument as in the above proof, we only need to control $\sum_{i \in B_\eta} ||C(\rp_i)||_{L^2(B_i)}^2$ and we obtain
\begin{proposition}
For all $q \geq 2$ and all divergence-free $\phi$ in $C^\infty_0(\R^3)$, 
\begin{align*}
    \underset{N \rightarrow + \infty}{\limsup} \, \Big| 2 \mu \int_{\R^3} D(\phi_N):D(\phi) \Big| \leq C  \lambda^{1-\frac{1}{q}} \eta^{\frac{\alpha}{2}}||D(\phi)||_{L^q(\R^3)}.
\end{align*}
\end{proposition}
The last contribution of $\phi_N^{app}$ was already studied in \cite{GVRH} and yields
\begin{proposition}
For all $q \geq 2$ and all divergence-free $\phi$ in $C^\infty_0(\R^3)$, 
\begin{align*}
    \underset{N \rightarrow + \infty}{\limsup} \, \Big| 2 \mu \int_{\R^3} D(\phi_N^{app}):D(\phi) \Big| \leq C (\lambda \eta^\alpha )^{\frac{q-1}{q}} ||D(\phi)||_{L^q(\R^3)}.
\end{align*}
\end{proposition}
\begin{proof}
For the sake of completeness, we remind here the argument employed in \cite{GVRH}. Following some tedious computations, $\phi_N^{app}$ satisfies the following equation in the sense of distributions
\begin{align*}
    - \mu \Delta \phi_N^{app} + \nabla \pi_N^{app} = \mu S x s^1  = - \, \div \big(  \mu \sum_{i \in B_\eta} C(\rp_i) 1_{B_i} \big) \quad \text{in} \quad \R^3.
\end{align*}
Testing this equation against a divergence-free smooth field $\phi$ yields 
\begin{align*}
    2 \mu \int_{\R^3} D(\phi_N^{app}):D(\phi) = \mu \sum_{i \in B_\eta}\int_{B_i} C(\rp_i) : D(\phi)
\end{align*}
and applying Hölder inequality gives
\begin{align*}
    \Big|  5 \mu \sum_{i \in B_\eta}\int_{B_i} D(\phi):C(\rp_i) \Big| &  \leq C \Big( \sum_{i \in B_\eta} \int_{B_i} |D \phi(x)|^q  \, \d x\Big)^{\frac{1}{q}} \Big( \sum_{i \in B_\eta} \int_{B_i} |C(\rp_i)|^{\frac{q}{q-1}} \Big)^{\frac{q-1}{q}}\\
    & \leq C (\lambda \eta^\alpha)^{\frac{q-1}{q}}||D(\phi)||_{L^q(\R^3)} \quad \text{owing to \eqref{assumption::empirical_measure}}.
\end{align*}
\end{proof}

\subsubsection{Proof of Theorem \ref{main_theorem relaxed duality}}

To sum up the different splittings we used on $v_N$, we wrote $v_N = \: v_N^1 \:+\: \psi_N \:+\: \phi_N \:- \:\phi_N^
{app}$ and we analysed separately each contribution. All in all, we proved that for any smooth divergence-free functions $\phi$, we have for any $q \geq 2$
\begin{multline}
        \Big|2 \mu \int_{\R^3} D(v_N) : D(\phi) \Big| \leq C \big( \lambda^{2-\frac{1}{q}} \eta^{-3}  \:+\: \lambda^{\frac{2}{3}-\frac{1}{q}} \eta^{1+\frac{\alpha}{2}- \frac{\alpha}{q}} \:+\: \lambda^{2-\frac{2}{q}} \eta^{-3+\frac{\alpha}{2}+\frac{3}{q}-\frac{\alpha}{q}} \\
        +\: \lambda^{1-\frac{1}{q}} \eta^{\frac{\alpha}{2}}  \:+ \: (\lambda \eta^\alpha)^{\frac{q-1}{q}} \big) ||D(\phi)||_{L^q(\R^3)}.
    \label{inegalité conclusion dualité}
\end{multline}
We can now turn to the proof of Theorem \ref{main_theorem relaxed duality}.
\begin{proof}[Proof of Theorem  \ref{main_theorem relaxed duality}]

As the sequence $u_N= \: u_N^{\frak{p}} \:+\: u_N^{app}\:+\:v_N$ is bounded in $L^p_{loc}(\R^3)$ for $1<p<\frac{3}{3}$, there exists an accumulation point $u_\lambda$ of $(u_N)$ in $L^p_{loc}(\R^3)$. Up to an extraction $N_k$, there exists $u_\lambda^{\frak{p}}$ an accumulation point of $(u_N)$ in $\dot{H}^1(\R^3)$ such that 
\begin{align*}
    u_\lambda = \:u_\lambda^{\frak{p}} \:+ \: w_0 \:+\: \underset{k \rightarrow + \infty}{\lim} \, v_{N_k},
\end{align*}
where $w_0$ is the limit of $u_N^{app}$ in $L^p_{loc}(\R^3)$, see Proposition \ref{convergence u_app^N}.\\

\noindent Following \textit{mutatis mutandis} the proof of Theorem \ref{main_theorem}, but replacing Theorem \ref{thm::DGV_Mecherbet} with Theorem \ref{theoreme DGV Höfer ordre 1},  we obtain the following variational formulation verified by $u_\lambda$ for any divergence-free $\phi$ in $C^\infty_0(\R^3)$
\begin{align*}
     2  \mu \int_{\R^3} (1+\frac{5}{2} \rho \lambda) D(u_\lambda):D(\phi) = \int_{\R^3} g \cdot \phi + \lambda \int_{\R^3}  \sigma_1 : D(\phi) + \langle \mR_\lambda| \phi \rangle
\end{align*}
and we only need to analyse the following remainder 
\begin{multline*}
   \langle \mR_\lambda | \phi \rangle  :=\langle \mR_\lambda^p | \phi \rangle + 2  \mu \int_{\R^3} (\frac{5}{2} \rho \lambda )  D(w_0): D(\phi) \\
   + \underset{k \rightarrow + \infty}{\lim} \, 2 \mu \int_{\R^3} (\frac{5}{2} \rho \lambda )  D(v_{N_k}): D(\phi) + \underset{k \rightarrow + \infty}{\lim} \, 2 \mu\int_{\R^3} D(v_{N_k}): D(\phi).
\end{multline*}
It was already verified (see \eqref{inequality w_0 in main theorem}) that 
\begin{align*}
    \Big| 2 \mu \int_{\R^3} (\frac{5}{2} \rho \lambda)  D(w_0): D(\phi) \Big| \leq C \lambda^{2} ||D(\phi)||_{L^q(\R^3)}.
\end{align*}
For any $q \geq q_{min}^{\frak{p}}$, we have $|\langle \mR_\lambda^p | \phi \rangle| \leq C \lambda^{1+\delta^{\frak{p}}}$ where $q_{min}^{\frak{p}}, \ \delta^{\frak{p}}$ and $\eta$ are specified in Remark \ref{remarque exposents}. We know from \eqref{inegalité conclusion dualité} that
\begin{align*}
    \Big| \underset{k \rightarrow + \infty}{\lim} \, 2 \mu \int_{\R^3} D(v_{N_k}): D(\phi)\Big| \leq C \sum_{k=1}^5 C_k(\lambda) ||D(\phi)||_{L^q(\R^3)}
\end{align*}
where 
\begin{align*}
     &C_1(\lambda) := \lambda^{2-\frac{1}{q}} \eta^{-3}, && C_2(\lambda) :=\lambda^{\frac{2}{3}-\frac{1}{q}} \eta^{1+\frac{\alpha}{2} - \frac{\alpha}{q}},\\
    & C_3(\lambda) :=\lambda^{2-\frac{2}{q}} \eta^{-3+\frac{\alpha}{2}+\frac{3}{q}-\frac{\alpha}{q}}   , && C_4(\lambda) := \lambda^{1-\frac{1}{q}} \eta^{\frac{\alpha}{2}} ,   \\
     &C_5(\lambda)  :=  (\lambda \eta^\alpha)^{\frac{q-1}{q}}.
\end{align*}
Of course, a similar estimate with an additional $\lambda$ holds for $2 \mu \int_{\R^3} (\frac{5}{2} \rho \lambda )  D(v_{N_k}): D(\phi)$. 
Choosing $\eta= \lambda^{\theta}$ with $\theta= \frac{2q}{6(q-1)+(q-2) \alpha}$, some tedious computations yields that there exists $q_\alpha$, only depending on $\alpha$, such that for all $q \geq q_\alpha$, one can find $\delta > 0$ such that 
\begin{align*}
    \sum_{k=1}^5 C_k(\lambda) \leq C \lambda^{1+\delta}.
\end{align*}
With the choice $q_{min} = \max(q_\alpha, q_{min}^{\frak{p}})$, we have simultaneously the above inequality as well as the needed passive system estimate $|\langle \mR_\lambda^p | \phi \rangle| \leq C \lambda^{1+\delta^{\frak{p}}} ||D(\phi)||_{L^q(\R^3)}$ . A density argument on the above estimates concludes the proof of the result.
\end{proof}

\section*{Acknowledgements}
I would like to acknowledge my PhD advisor David Gérard-Varet for his confidence and guidance, as well as his scientific suggestions and careful rereadings on this work. I also express my gratitude to Richard Höfer for his comments that greatly improved the manuscript. I acknowledge the support of the DIM Math Innov de la Région Ile-de-France.

\begin{figure}[H]
    \centering
    \includegraphics[scale=0.3]{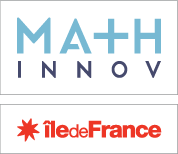}
    \label{fig:my_label}
\end{figure}

\appendix

\section{Proof of Lemma \ref{lemma::construction de v[p]}} 
\label{appendix::lemma_v[p]}
\begin{proof}[Proof of Lemma \ref{lemma::construction de v[p]}]
In order to compute $v[\rp]$ and $p[\rp]$, we use Stokes equations linearity to decompose $v[\rp]:=  v_1 + v_{2}+ v_3$.  The first term is $v_1$, which represents the contribution from the point force source, satisfies in the sense of distributions
 \begin{equation*}
    \left\{
      \begin{aligned}
        - \mu\Delta v_1 + \na p_1 & = - k_f  \delta(x- a\beta\rp) \, \rp &&\ \text{in} \ \R^3 \setminus B,\\
        \div v_1 & = 0 &&\ \text{in} \ \R^3 \setminus B,\\
        v_1 & =0 &&\text{ on } \partial B,\\
        \underset{|x| \rightarrow + \infty}{ \lim} v_1(x) & = 0.
      \end{aligned}
    \right.
\end{equation*}
The solution can be computed using the Oseen tensor and the \textit{Method of Images}. Roughly speaking, this method corresponds to consider a virtual point force inside the particle, the intensity and the position of which is chosen to respect the zero velocity condition on $\partial B$. Using computations from  \cite[Chapter 10]{KK}, we get
\begin{align*}
    v_1(x)  = &   - \: k_f U(x-a\beta \rp) \, \rp \:+\: k_f \gamma_1 U(x-a\beta^{-1} \rp) \, \rp \\
    &  +\: k_f \gamma_2 a \na U(x-a\beta^{-1} \rp) (\rp \otimes \rp- \frac{Id}{3}) \:+\:k_f \gamma_3 a^2 \Delta U(x-a\beta^{-1} \rp) \, \rp,
\end{align*}
with $\gamma_1$, $\gamma_2$ and $\gamma_3$ coefficients depending on the point force proximity through $\beta$ by
\begin{align*}
     \gamma_1 = \frac{3 \beta^{-1}}{2}  - \frac{\beta^{-3} }{2}, \quad \gamma_2 =  \beta^{-2} - \beta^{-4} \quad \text{and} \quad \quad \gamma_3 = \frac{\beta^{-1}}{4}(1-\beta^{-2})^2.
\end{align*}
The associated pressure $p_1$ is computed using Stokes equations linearity and the well-known pressure associated to the \textit{Stokeslet} $U$, the \textit{Stokeslet doublet} $\na U$ and $\Delta U$:
\begin{multline*}
        p_1(x)  = \frac{k_f}{4 \pi} \Big[ \frac{\rp \cdot(x-a \beta \rp)}{|x- a \beta \rp|^3} \:+\: \frac{\rp \cdot(x-a \beta^{-1} \rp)}{|x- a \beta^{-1} \rp|^3}  \\
     + \: 3 \gamma_2 a \frac{(\rp \otimes \rp - \frac{Id}{3}):(x- a \beta^{-1} \rp)\otimes(x- a \beta^{-1} \rp)}{|x- a \beta^{-1} \rp|^5}  \: +3 \: \gamma_3 a^2\frac{\rp \cdot(x-a \beta^{-1} \rp)}{|x- a \beta^{-1} \rp|^5} \Big].
\end{multline*}
Through integration by parts, we can compute the force and torque on the sphere associated to this solution 
\begin{equation}
    \int_{\partial B}\sigma(v_1,p_1)n  = - k_f \gamma_1 \,\rp \quad \text{and} \quad \int_{\partial B}\sigma(v_1,p_1)n \times x     = k_f \gamma_1 \, \rp \times \rp = 0.
    \label{force and torque v_1}
\end{equation}
Note that using tedious computations from \cite[Section 10.2.1]{KK}, one can also compute the following \textit{Stresslet}
\begin{align*}
    \int_{\partial B } [\sigma(v_1,p_1)^+ n](y) \otimes y  \, \d s (y) = k_f \Big( - \frac{5}{2} \beta^{-2} + \frac{3}{2} \beta^{-4} \Big) \, \rp \otimes \rp.
\end{align*} \medskip 

 \noindent The second term $v_{2}$ comes from the particles reaction to the pushing/pulling and solves
 \begin{equation*}
    \left\{
      \begin{aligned}
        - \mu\Delta v_{2} + \na p_2 & = 0  &&\ \text{in} \ \R^3 \setminus B,\\
        \div v_{2} & = 0 &&\ \text{in} \ \R^3 \setminus B,\\
        v_{2} & = U_2 &&\text{ on } \partial B,\\
        \int_{\partial B} (\sigma_{\mu}(v_1,p_1) +\sigma(v_{2},p_2)) n   \, \d s   + k_f \rp & =0 \\
        \underset{|x| \rightarrow + \infty}{ \lim} v_{2}(x) &  = 0.
      \end{aligned}
    \right.
\end{equation*}
Using Stokes law for a sphere translating at speed $U_2$ in a Stokes flow and the force expression \eqref{force and torque v_1}, we get that $\int_{\partial B}\sigma(v_{2},p_2)n = -6 \pi  \mu a U_2 = (\gamma_1-1) k_f \, \rp$ which leads to $U_2=\frac{k_f (1-\gamma_1)}{6 \pi \mu a} \rp$. Note that this velocity direction is consistent with Figure \ref{Geometry 1} and \ref{Geometry 2} as a pusher particle with a negative $k_f$ swims in the direction $- \rp$, whereas a puller one ($k_f>0$) swims in the direction $\rp$. Using computations from \cite{guazzelli2011physical}, we find 
$$v_{2}(x)= k_f (1-\gamma_1) U(x) \, \rp + k_f (1-\gamma_1)\frac{a^2}{|x|^2} \Tilde{U}(x) \, \rp, \quad p_2(x)=  k_f (1-\gamma_1) \frac{ x \cdot \rp}{4 \pi |x|^3},$$
where we remind that $\Tilde{U}(x)= \frac{1}{8 \pi \mu} ( \frac{Id}{3 |x|} - \frac{x \otimes x}{|x|^3})$. The torque $\int_{\partial B}\sigma(v_2,p_2)n \times x$ produced by this solution is also equal to zero. A simple computation shows that the associated Stresslet $\int_{\partial B} \sigma(v_2,p_2)^+ n \otimes y \, \d s(y)$ is zero. \medskip

\noindent Lastly, $v_3$ solves the rotation problem 
 \begin{equation*}
    \left\{
      \begin{aligned}
        - \mu\Delta v_{3} + \na p_3 & = 0  &&\ \text{in} \ \R^3 \setminus B,\\
        \div v_{3} & = 0 &&\ \text{in} \ \R^3 \setminus B,\\
        v_{3}(x) & = V_3 \times x &&\text{ on } \partial B,\\
        \sum_{k=1}^3 \int_{\partial B}\sigma(v_k,p_k) n \times x   \, \d s   & =0 \\
        \underset{|x| \rightarrow + \infty}{ \lim} v_{3}(x) & = 0.
      \end{aligned}
    \right.
\end{equation*}
Since we have  $\int_{\partial B}\sigma(v_1,p_1) n \times x= \int_{\partial B}\sigma(v_2,p_2) n \times x=0$, computations from \cite{guazzelli2011physical} show that rotation speed $V_3$ is zero, which finally gives that $v_3$ and $p_3$ are both zero.
\noindent Altogether, the active contribution boils down to 
\begin{align*}
     v[\rp](x) = & \, v_1(x)+ v_{2}(x) \\
     = & - k_f U(x-a\beta \rp) \, \rp + k_f \gamma_1 U(x-a\beta^{-1} \rp) \, \rp + k_f \gamma_2 a \na U(x-a\beta^{-1} \rp) \big(\rp \otimes \rp- \frac{Id}{3}\big)\\
    &+ k_f (\gamma_1-1) U(x) \, \rp + k_f(\gamma_1-1) \frac{a^2}{|x|^2} \Tilde{U}(x) \, \rp, \quad x \in \R^3 \setminus B.
\end{align*}
 We need to find a good way to extend $v[\rp]$ inside the ball $B$. To do so, notice that 
\begin{equation*}
     v[\rp](x)=v_1(x)+ v_{2}(x)  = U_2= \frac{ k_f (1-\gamma_1)}{6 \pi \mu a}  \, \rp \quad \text{for} \quad x \in \partial B,
\end{equation*}
and we extend $v[\rp]$ by this solid field inside the ball. Note that this extension does not create a jump at the boundary and thus $v[\rp]$ lies in $\dot{W}_{loc}^{1,r}(\R^3)$ with $1 \leq r < 3$ using the jump formula.
\end{proof}

\section{Proof of Lemma \ref{lemma comparaison série intégrale}}
\label{appendix::lemma serie}
\begin{proof}[Proof of Lemma \ref{lemma comparaison série intégrale}]
The proof is similar for the two inequalities. We will only deal with  \eqref{lemme ine 1}. Let $(i,j)$ fixed in   $G_\eta \times G_\eta$. For any $y \in B(y_j, \frac{\eta}{2})$, we have straightforwardly 
\begin{align*}
    \frac{1}{( \eta + |y_i - y_j|)^{4}} \leq \frac{1}{(\frac{\eta}{2}  + |y_i - y|)^{4}}.
\end{align*}
Integrating over $B(y_j, \frac{\eta}{2})$ and summing on the $j$ in $G_\eta$, we obtain
\begin{align*}
    \sum_{j \in G_\eta}   \frac{1}{( \eta + |y_i - y_j|)^{4}} & \leq C \eta^{-3} \sum_{j \in G_\eta} \int_{B(y_j,\frac{\eta}{2})} \frac{1}{(\frac{\eta}{2}+ |y_i - y|)^{4}} \, \d y\\
    & \leq C \eta^{-3} \int_{\R^3 \setminus B(y_i,\frac{\eta}{2})} \frac{1}{(\frac{\eta}{2} + |y_i - y|)^{4}} \, \d y
\end{align*}
using that $\cup_{j \in G_\eta} B(y_j, \frac{\eta}{2}) $ is a disjoint family of balls strictly included in $\R^3 \setminus B(y_i,\frac{\eta}{2})$ since $|y_i - y_j| \geq \eta$ for any $j$ in $G_\eta$. Using a change of variables, we wish to estimate 
\begin{align*}
    \int_{\R^3 \setminus B(0,\frac{\eta}{2})} \frac{1}{(\frac{\eta}{2} + | y|)^{4}} \, \d y & = C \int_{r=\frac{\eta}{2}}^{+ \infty} \frac{r^2}{(\frac{\eta}{2} + r)^{4}} \, \d r\\
    & \leq C \int_{r=0}^{+ \infty} \frac{ (r- \frac{\eta}{2})^2}{r^4} \, \d r \\
    & \leq C \eta^{-1}
\end{align*}
which yields the expected result. 
\end{proof}
{\footnotesize 
\bibliographystyle{siam}

 \bibliography{biblio}
 }
 
 \end{document}